\documentclass[journal,twoside,web]{ieeecolor}
\usepackage[colorlinks=true,citecolor=blue,linkcolor=blue,urlcolor=blue]{hyperref}
\usepackage{generic}         
\usepackage{booktabs}  
\usepackage{amsmath,amssymb,amsfonts}
\usepackage{dsfont}
\usepackage{graphicx}

\usepackage{bbm}
\usepackage{xcolor}
\usepackage{algorithm}
\usepackage{algpseudocode}
\graphicspath{{./}{./figures/}}

\usepackage{tikz}
\usepackage{matlab-prettifier}
\usepackage[mathscr]{euscript}
\usepackage{url}
\usepackage{cite}
\usepackage{mathtools}

\usepackage{subcaption}
\usepackage{bbm}

\usepackage{etoolbox}

\newtheorem{theorem}{\normalfont\bfseries Theorem}
\newtheorem{corollary}{\normalfont\bfseries Corollary}
\newtheorem{lemma}{\normalfont\bfseries Lemma}
\newtheorem{proposition}{\normalfont\bfseries Proposition}
\newtheorem{problem}{\normalfont\bfseries Problem}
\newtheorem{assumption}{\normalfont\bfseries Assumption}

\newtheorem{example}{Example}
\newtheorem{remark}{Remark}

\AtBeginEnvironment{theorem}{\itshape}
\AtBeginEnvironment{corollary}{\itshape}
\AtBeginEnvironment{lemma}{\itshape}
\AtBeginEnvironment{proposition}{\itshape}
\AtBeginEnvironment{problem}{\itshape}
\AtBeginEnvironment{assumption}{\itshape}

\newcommand{\calA}{{\mathcal A}}
\newcommand{\calB}{{\mathcal B}}
\newcommand{\calC}{{\mathcal C}}
\newcommand{\calD}{{\mathcal D}}
\newcommand{\calE}{{\mathcal E}}
\newcommand{\calF}{{\mathcal F}}

\newcommand{\calH}{{\mathcal H}}
\newcommand{\calI}{{\mathcal I}}
\newcommand{\calJ}{{\mathcal J}}

\newcommand{\calN}{{\mathcal N}}
\newcommand{\calO}{{\mathcal O}}
\newcommand{\calP}{{\mathcal P}}

\newcommand{\calR}{{\mathcal R}}

\newcommand{\calX}{{\mathcal X}}
\newcommand{\calY}{{\mathcal Y}}

\newcommand{\scrB}{{\mathscr B}}

\newcommand{\scrH}{{\mathscr H}}

\newcommand{\bbE}{{\mathbb E}}

\newcommand{\bbK}{{\mathbb K}}

\newcommand{\bbP}{{\mathbb P}}
\newcommand{\bbQ}{{\mathbb Q}}
\newcommand{\bbR}{{\mathbb R}}

\newcommand{\bbU}{{\mathbb U}}

\newcommand{\bbZ}{{\mathbb Z}}

\newcommand{\diag}{\mathop{\rm diag}\nolimits}

\newcommand{\bfu}{{\mathbf u}}
\newcommand{\bfv}{{\mathbf v}}

\newcommand{\bfP}{{\mathbf P}}
\newcommand{\bfQ}{{\mathbf Q}}
\newcommand{\bfR}{{\mathbf R}}
\newcommand{\bfK}{{\mathbf K}}

\newcommand{\rmd}{{\rm d}}

\newcommand{\rmx}{{\rm x}}

\newcommand{\sfM}{{\sf M}}

\newcommand{\bfone}{\mathbbm{1}}

\newcommand{\Ee}{{\rm e}}   

\newcommand{\bbra}[1]{\ensuremath{[\![#1]\!]} }  % [[  ]]

\newcommand{\kl}[2]{D_{\textrm{\rm KL}} ({#1}\|{#2})}
\newcommand{\klg}[2]{\calD ({#1}\|{#2})}

\newcommand{\wtilde}[1]{\widetilde{#1}}

\newcommand{\jref}{{\bar{\mathbf{Q}}}}

\newcommand{\x}{{\xi}}
\newcommand{\kil}{{\rm kill}}
\newcommand{\FK}{{\rm FK}}

\DeclareMathOperator{\rank}{rank}

\DeclareFontFamily{U}{mathx}{\hyphenchar\font45}
\DeclareFontShape{U}{mathx}{m}{n}{
      <5> <6> <7> <8> <9> <10>
      <10.95> <12> <14.4> <17.28> <20.74> <24.88>
      mathx10
      }{}
\DeclareSymbolFont{mathx}{U}{mathx}{m}{n}
\DeclareFontSubstitution{U}{mathx}{m}{n}
\DeclareMathAccent{\widecheck}{0}{mathx}{"71}
\DeclareMathAccent{\wideparen}{0}{mathx}{"75}

\newcommand{\ft}{N}

\newcommand{\cyan}[1]{{\color{black}{#1}}} 
 
\newcommand{\red}[1]{{\color{black}{#1}}}

\newcommand{\dlt}[1]{}

\newcommand{\schr}{{Schr\"{o}dinger}}

\begin{document}

\bstctlcite{}
\title{
Control Realization of \schr~Bridges over Discrete-Time Nonlinear Systems via Maximum-Entropy Distribution Steering
}
\author{Kaito Ito\thanks{Kaito Ito is with the Department of Information Physics and Computing, The University of Tokyo, Tokyo 113-8654, Japan {\tt \small kaito@g.ecc.u-tokyo.ac.jp}.}
and Kenji Kashima\thanks{Kenji Kashima is with the Graduate School of Informatics, Kyoto University, Kyoto 606-8501, Japan {\tt \small kk@i.kyoto-u.ac.jp}.}
\thanks{This work was supported in part by JSPS KAKENHI Grant Number JP24K17297.}
}

\maketitle
\thispagestyle{empty}

\begin{abstract}
\red{This paper shows how a \schr~bridge (SB) structure emerges from maximum-entropy (MaxEnt) distribution steering of discrete-time nonlinear systems. We also establish that the resulting bridge can be realized by an optimal control policy.}
We consider general state and input costs, non-Gaussian endpoint distributions, and input constraints. 
Assuming the existence of a feasible policy with finite objective value, we show that the state-path measure under a MaxEnt optimal control policy is the unique SB. The reference measure of this SB is induced by an optimal policy of the corresponding terminal-unconstrained problem.
Moreover, we give an equivalent characterization of the optimal state-path measure using a killed reference and its normalized Feynman--Kac measure, thereby interpreting the state cost as killing.
We then characterize the finite-cost feasibility by the existence of an endpoint coupling having finite Kullback--Leibler divergence from the reference endpoint distribution. For compactly supported marginals with bounded target density, we derive reachability-based sufficient conditions guaranteeing this feasibility. We further provide specializations to fully actuated control-affine and reachable linear systems.
\end{abstract}


\section{Introduction}
\label{sec:introduction}

\IEEEPARstart{T}{his} paper studies discrete-time \schr~bridges (SBs) on continuous spaces.
The SB problem originates from an inquiry by \schr~\cite{Schrodinger1931,Schrodinger1932}: given marginals and prior dynamics, what is the most probable trajectory connecting the two prescribed marginals? In a standard modern formulation, it is posed as the minimization of the Kullback--Leibler (KL) divergence between the path measure over trajectories and the reference measure induced by the prior~\cite{Leonard2014}.
SB problems can be understood as entropy-regularized optimal transport~\cite{Peyre2019} and have recently found applications in generative modeling~\cite{De2021}, image restoration~\cite{Liu23ai}, system identification from snapshot data~\cite{Morimoto2025linear}, and trajectory inference from single-cell data~\cite{Chen2023deep}.

The SB problem is inherently an optimal distribution steering problem that transports the state distribution from the initial distribution to the desired distribution, and there is a rich literature on the control-theoretic viewpoint of SBs, especially for continuous-time systems~\cite{Follmer2006random,Mikami1990variational,Pra1990,Chen2016relation}; see~\cite{Chen2021liaison} for a comprehensive review and \cite{Mikami2021stochastic} for a monograph on SBs and stochastic control with fixed marginals.
Under the distribution control minimizing a cumulative quadratic control cost, the state process follows an SB, where the prior is induced by the uncontrolled dynamics driven by noise.
When the cost functional contains state costs, the distribution steering problem is equivalent to an SB problem with killing of probability mass~\cite{Wakolbinger1989,Nagasawa1989stochastic,Pra1990,Teter2024schr}.

Despite extensive studies on the relationships between SBs and control, comparatively limited attention has been devoted to discrete-time settings~\cite{Beghi1996,Levy1997discrete,Vladimirov2015state,Chen2016robust,Lambert2025lqr}. The work~\cite{Ito2022maximum} established that, for possibly degenerate discrete-time linear systems with Gaussian driving noise and Gaussian endpoint marginals, the SB is equivalent to optimal distribution steering with an entropy-regularized quadratic input cost. 
Entropy-regularized, or maximum-entropy (MaxEnt), distribution steering with state and input costs was subsequently studied in the linear-quadratic-Gaussian (LQG) setting~\cite{Ito2025maximum}, \red{and an SB characterization was established under invertibility assumptions on the system and input matrices~\cite{Zhan2026}.
Beyond the LQG setting, however, it remains unclear whether and how MaxEnt distribution steering can be characterized as an SB for nonlinear dynamics, non-Gaussian endpoint marginals, and general state and input costs.
This paper addresses these questions.
}

{\it Contributions:}
The main contributions of this paper are summarized as follows:
\begin{enumerate}
    \item We establish the equivalence between MaxEnt distribution steering of deterministic discrete-time nonlinear systems and an SB, allowing non-Gaussian marginals, general state and input costs, and input constraints (Theorem~\ref{thm:maxent_solves_SB}). The SB reference is induced by the optimal state process of the corresponding terminal-unconstrained control problem. We also characterize optimal Markov policies that realize the bridge (Propositions~\ref{prop:opt_policy_condi}--\ref{prop:factorization_sufficient}). Moreover, we interpret the state cost as a killing mechanism by showing that the optimal state-path measure is also the unique minimizer of the SB problems defined by the killed reference and its normalized Feynman--Kac reference (Theorem~\ref{thm:killed_FK_equiv}). 
    \item We prove that finite-cost feasibility of MaxEnt distribution steering is equivalent to the existence of an endpoint coupling with finite KL divergence from the reference endpoint distribution (Corollary~\ref{cor:finite_cost}). 
    This condition, together with Theorem~\ref{thm:maxent_solves_SB}, also guarantees the existence of an optimal control policy.
    For compactly supported marginals and a bounded target density, we derive reachability-based sufficient conditions for the finite-cost feasibility (Theorem~\ref{thm:uniform-reachability-finite-entropy}).
    We further provide blockwise, fully actuated control-affine, and linear-system specializations (Corollaries~\ref{cor:reachability_block}--\ref{cor:linear_reachability}).
\end{enumerate}

\red{These results reveal a discrete-time SB structure
that emerges directly from the control problem.
Because our construction starts from a Gibbs reference on the control input,
it accommodates general deterministic nonlinear dynamics, as well as
general input costs and input constraints.
This contrasts with the Girsanov-based continuous-time control interpretation of SBs, which relies on control-affine drift perturbations. In that setting, the path-space KL divergence induces a quadratic input cost~\cite{Follmer2006random,Pra1990}.
Although our SB reference is derived from the associated MaxEnt control problem, its induced path measure retains the noise-driven structure familiar from continuous-time SBs.
This construction of the SB reference provides both an admissible control realization of the bridge and access to efficient computational methods for SBs including the Sinkhorn algorithm (see Subsection~\ref{subsec:numerical}).}

{\it Related work:}
\red{An early study \cite{Beghi1996} considered discrete-time SBs for Markov processes with transition densities and
showed that the SB is again Markov. The Gaussian case was further studied in \cite{Levy1997discrete}.
Subsequent works studied optimization of driving-noise distributions
under prescribed endpoint state distributions
\cite{Vladimirov2015state}, and finite-state SB formulations for
network transport~\cite{Chen2016robust}.
More recently, \cite{Lambert2025lqr} considered SBs with LQR-type path costs and Gaussian marginals.
However, these discrete-time SB formulations do not in general establish an equivalence with optimal control for prior dynamics or provide an admissible control policy that realizes the SB.
Our formulation addresses this gap by deriving the SB reference directly from the MaxEnt control problem, which yields both the equivalence and an admissible control realization.

In continuous time, state costs have been related to killed reference
processes through Feynman--Kac weighting since at least
\cite{Pra1990}. Our reference construction reveals the corresponding discrete-time
structure: the state cost appears as a survival weighting of the
noise-driven reference, while the resulting normalized Feynman--Kac
measure yields the same SB.

}

For distribution steering without entropy regularization, suitable point-to-point reachability, together with standard compactness and continuity assumptions, guarantees feasibility under stochastic feedback~\cite{Elamvazhuthi2018optimal}.
For finite-cost SBs, mere reachability is insufficient. At least one endpoint coupling must also have finite KL divergence with respect to the reference. 
Our sufficient conditions ensure this by requiring a uniformly nonvanishing family of regular steering inputs that remains in the interior of the admissible input set.

\textit{Organization:}
The remainder of this paper is organized as follows. In Section~\ref{sec:formulation}, we formulate the SB and MaxEnt distribution steering problems.
In Section~\ref{sec:SB_maxent_eq}, we establish the SB characterization of the MaxEnt distribution steering and characterize optimal policies.
In Section~\ref{sec:killing}, we give an interpretation of the state cost through a killed reference measure and a normalized Feynman--Kac measure.
In Section~\ref{sec:reachability}, we derive reachability-based sufficient conditions for the existence of finite-KL endpoint couplings.
Section~\ref{sec:conclusion} concludes the paper.

\textit{Notations:}
Let $ \bbR $ denote the set of real numbers.
Denote the sets of positive and nonnegative integers by
$\mathbb{Z}_{>0}$ and $\mathbb{Z}_{\geq 0}$, respectively.
For $ k, \ell \in \bbZ_{\ge 0} $ such that $ k \le \ell $, a sequence $ \{ k,\ldots,\ell\} $ is denoted by $ \bbra{k,\ell} $.
A sequence of vectors $ \{x_{k},\ldots,x_\ell\} $ is denoted by $ x_{k:\ell} $.
For a symmetric matrix $ A $, we write $ A \succ 0 $ (resp. $ A\succeq 0 $) if $ A $ is positive definite (resp. semidefinite). 
The identity matrix is denoted by $ I $, and its dimension depends on the context.
The transpose of the inverse of an invertible matrix $ A $ is denoted by $ A^{-\top} $.
The Euclidean norm is denoted by $ \|\cdot \| $.
The set of probability measures on a set $ \calX $ is denoted by $ \calP(\calX) $.
The expectation is denoted by $ \bbE $.
\cyan{The KL divergence between probability measures $ \bbP $ and $ \bbQ $ is denoted by $ \kl{\bbP}{\bbQ} $.}
A random variable and its realization may share the same symbol.
For a measurable map $ T : \calX \to \calY $, $ T_\# \mu $ denotes the pushforward of a measure $ \mu $, defined by $ (T_\#\mu) (S) = \mu(T^{-1} (S)) $ for measurable $ S \subseteq \calY $.
For a topological space $ S $, $ \scrB(S) $ denotes the Borel $ \sigma $-algebra on $ S $.
Denote by $\mathbf{1}$ an all-ones vector of
dimension determined by context, whereas $\bfone_S$ denotes
the indicator function of a set $S$.

\section{Problem formulation}\label{sec:formulation}
We first introduce the MaxEnt distribution steering problem, and then recall the SB problem.

\subsection{MaxEnt distribution steering}
In this paper, we consider a deterministic discrete-time nonlinear system
\begin{align}
    x_{k+1} &= f_k (x_k, u_k), \qquad k \in \bbZ_{\ge 0}, \label{eq:system}\\
    x_0 &\sim \mu_0 ,
\end{align}
where $ \{x_k\} $ is an $ \bbR^n $-valued state process, $ \{u_k\} $ is a $ \bbU $-valued control process, $ f_k : \bbR^n \times \bbU \rightarrow \bbR^n $ is Borel measurable, and $ \mu_0 \in \calP(\bbR^n) $ is the initial state distribution.
The input set $ \bbU \subseteq \bbR^m $ is a nonempty Borel set satisfying
$
    \mu_L^m (\bbU) > 0
$, 
where $ \mu_L^m $ denotes the Lebesgue measure on $ \bbR^m $.\footnote{The SB characterization in Sections~\ref{sec:SB_maxent_eq} and \ref{sec:killing} extends to discrete input spaces by replacing the Lebesgue measure with the counting measure.}
To formulate the MaxEnt distribution steering problem, we introduce state and input cost functions $ q_k $ and $ \ell_k $, respectively, under the following assumption.
\begin{assumption}[Costs]\label{ass:cost}
    For any $ k \in \bbra{0,\ft-1} $, $ q_k : \bbR^n \rightarrow [0,\infty) $ and $ \ell_k : \bbU \rightarrow \bbR $ are Borel measurable and
    \begin{equation}
        0 < Z_{\ell,k} := \int_{\bbU} \Ee^{-\ell_k (u)} \rmd u < \infty .
    \end{equation}
\end{assumption}
\vspace{0.2cm}

The above integrability condition allows us to define the Gibbs reference input measure
\begin{equation}\label{eq:gibbs_input}
    R_k (\rmd u ) := (Z_{\ell,k})^{-1} \Ee^{-\ell_k(u)} \mu_L^m|_\bbU (\rmd u) ,
\end{equation}
where $ \mu_L^m|_\bbU $ is the restriction of the Lebesgue measure $ \mu_L^m $ to $ \bbU $. By construction, we have $ R_k (\bbU) = 1 $.

A causal control policy $ \pi = \{\pi_k\}_{k=0}^{\ft-1} $ consists of Borel stochastic kernels on $ \bbU $ given the history
$
    h_k := \{ x_0,u_0,\ldots,x_{k-1},u_{k-1},x_k \} 
$.
Under the probability measure $ \bbP^\pi $ induced by $ \pi $, the control input $ u_k $ is generated such that $ \bbP^\pi (u_k \in A \mid h_k) = \pi_k (A \mid h_k) $ for any Borel set $ A \subseteq \bbU $.
A causal policy is called a Markov policy if, for each $ k \in \bbra{0,\ft-1} $, $ \pi_k (\cdot \mid h_k) $ depends on the history $ h_k $ only through the current state $ x_k $, in which case we write $ \pi_k(\cdot \mid x_k) $.

The MaxEnt control objective is given by
\begin{equation}\label{eq:maxent_objective}
    J(\pi) := \bbE^\pi \! \left[ \sum_{k=0}^{\ft-1} \Bigl( q_k(x_k) + \kl{\pi_k(\cdot \mid h_k)}{R_k} \Bigr) \right] ,
\end{equation}
where $ \bbE^\pi $ denotes the expectation under the policy $ \pi $ and $ N \in \bbZ_{>0} $ is a finite horizon.
The KL term penalizes deviations of the policy $ \pi_k (\cdot \mid h_k) $ from the reference input measure $ R_k $.
Since $ q_k  $ is nonnegative, if $J(\pi)<\infty$, then, for each $k$, $\pi_k(\cdot\mid h_k)\ll R_k$ for almost every $ h_k $ under $\pi$.
Moreover, since $R_k\ll \mu_L^m|_{\bbU}$, the kernel $ \pi_k $ admits a conditional density $ \varpi_k (\cdot \mid h_k) $ with respect to the Lebesgue measure. If the expected input cost and the differential entropy $ \scrH(\varpi_k (\cdot\mid h_k)) := -\int_{\bbU} \varpi_k(u|h_k) \log \varpi_k (u|h_k) \rmd u $ are finite, the KL term can be written as
\begin{align*}
    &\kl{\pi_k(\cdot \mid h_k)}{R_k} \\
    &= \int_\bbU \ell_k (u) \varpi_k (u  \mid h_k) \rmd u - \scrH(\varpi_k(\cdot \mid h_k)) + \log Z_{\ell,k} .
\end{align*}
Therefore, up to an additive constant independent of $ \pi $, the objective $ J $ takes the entropy-regularized form
\[
\bbE^\pi \Biggl[ \sum_{k=0}^{\ft-1} \biggl( q_k (x_k) + \ell_k (u_k) - \scrH(\varpi_k (\cdot \mid h_k))   \biggr)  \Biggr] .
\]
This representation shows that the KL term combines the expected input cost with an entropy-regularization term.
Entropy regularization offers several benefits in reinforcement learning, including encouraging exploration; see, e.g.,~\cite{Levine2018reinforcement,Geist2019theory}.

The MaxEnt distribution steering problem seeks a control policy that steers the state from the initial distribution $ \mu_0 $ to the desired distribution $ \mu_N $ at the terminal time, while minimizing the cumulative cost.
Specifically, this problem is formulated as follows.
\begin{problem}[\bf MaxEnt distribution steering]\label{prob:MaxEnt}
	Given initial and target distributions $ \mu_0, \mu_\ft \in \calP(\bbR^n) $, find a control policy $ \pi = \{\pi_k\}_{k=0}^{\ft-1} $ that solves
	\begin{align}
            &\underset{\pi }{\textrm{\rm minimize}}  && J(\pi) \label{eq:cost} \\
			&{\textrm{\rm subject to}} && \eqref{eq:system}, \quad u_k \in \bbU  , \qquad \forall k \in \bbra{0,\ft-1}, \label{eq:input_constraint}\\
				& &&u_k \sim \pi_k (\cdot \mid h_k) , \qquad \forall k \in \bbra{0,\ft-1}, \label{eq:policy}\\
				& &&x_0 \sim \mu_0, \quad  x_N \sim \mu_\ft . \label{eq:cov_constraint}
	\end{align}
\end{problem}
\vspace{0.2cm}

\begin{remark}
    When the map $ u\mapsto f_k (x,u) $ is bijective, for Markov policies, the KL term in $ J $ is equivalently expressed as the KL divergence between the controlled and reference one-step state-transition distributions, recovering the continuous-space KL control in~\cite{Ito2022kullback}. The present input-space formulation does not require such invertibility and therefore also accommodates degenerate state transitions.
    \hfill $ \diamondsuit $
\end{remark}

\subsection{\schr~bridge}
In an SB problem, a reference measure over trajectories and marginal distributions at the endpoints of a time interval are given. The reference measure describes a prior stochastic evolution, which was originally modeled as Brownian motion~\cite{Schrodinger1931,Schrodinger1932} and has been extended to general diffusion processes~\cite{Follmer2006random,Mikami1990variational,Pra1990}.
Given prior dynamics, the SB problem seeks a probability distribution over trajectories that satisfies the prescribed endpoint marginals while remaining as close as possible to the reference evolution.
Specifically, the discrete-time SB problem is formulated as the minimization of the KL divergence between this trajectory distribution and the reference measure.
A probability measure on the state-trajectory space $ (\bbR^{n})^{N+1} $ is referred to as a path measure.
\begin{problem}[\bf Schrödinger bridge]
\label{prob:SB}
Given a reference path measure ${\mathbb{Q}}_{0:N} \in \calP((\bbR^{n})^{N+1})$ and marginals $ \mu_0, \mu_\ft \in \calP(\bbR^n) $, find a path measure $\mathbb{P}_{0:N} \in \calP((\bbR^{n})^{N+1}) $ that solves
\begin{align} 
    &\underset{\mathbb{P}_{0:N} \in \calP((\bbR^{n})^{N+1})}{\textrm{\rm minimize}} && \quad D_{\textrm{\rm KL}}\! \left(\mathbb{P}_{0:N} \middle\| {\mathbb{Q}}_{0:N}\right) \label{eq:SB_objective}\\
    &\quad \textrm{\rm ~subject to}  && (T_0)_{\#} \mathbb{P}_{0:N} = \mu_0, \ \ (T_\ft)_{\#} \mathbb{P}_{0:N} = \mu_\ft,\nonumber 
\end{align}
where $ T_k (x_0,\ldots,x_\ft) := x_k $.
\end{problem}

An optimal solution to Problem~\ref{prob:SB} is called an SB between $ \mu_0 $ and $ \mu_\ft $.
The work~\cite{Ito2022maximum} studied the SB problem where the reference $ \bbQ_{0:\ft} $ is given by the linear system $ x_{k+1} = A_k x_k + B_k w_k $ driven by the standard Gaussian noise $ \{w_k\} $ and marginals $ \mu_0 $ and $ \mu_\ft $ are Gaussian. The authors showed that the unique SB is the path measure generated by the MaxEnt distribution steering with $ q_k(x) \equiv 0 $, $ \ell_k (u) = \| u \|^2 /2 $, and $ \bbU = \bbR^m $.

\section{\schr~bridge characterization of MaxEnt distribution steering}\label{sec:SB_maxent_eq}
This section establishes the equivalence between MaxEnt distribution steering and the SB problem. We first reformulate the control objective as a KL divergence on the state-input space, then relate the resulting problem to an SB, and finally characterize optimal Markov policies.

\subsection{MaxEnt control as KL minimization}\label{subsec:kl_reformulation}

We begin by introducing an auxiliary MaxEnt control problem obtained by removing the terminal distribution constraint.
\begin{align}
        &\underset{\pi }{\textrm{minimize}}  && J(\pi) 
        &\textrm{subject~to} && \eqref{eq:input_constraint}\text{,~}\eqref{eq:policy}, ~ x_0 \sim \mu_0 . \label{prob:unconstrained}
\end{align}
We show that the optimal state-input distribution of the above problem serves as the reference measure in the KL reformulation of the MaxEnt control problem.

To characterize the optimal policy of Problem~\eqref{prob:unconstrained}, define the so-called desirability functions~\cite{Ito2022kullback,Todorov2009} by
\begin{align}
    \psi_\ft (x) &= 1 , \qquad  x \in \bbR^n , \label{eq:desire_boundary}\\
    \psi_k (x) &= \Ee^{-q_k(x)} \int_{\bbU} \psi_{k+1} (f_k(x,u)) R_k (\rmd u) ,\nonumber\\
     &\qquad\qquad\quad (k,x) \in \bbra{0,\ft-1}\times \bbR^n. \label{eq:desire_bellman} 
\end{align}
By backward induction, we have that for any $ x \in \bbR^n $,
\begin{equation}\label{eq:psi_bounded}
0 < \psi_k (x) \le 1  .
\end{equation}
We will see that
\begin{equation}\label{eq:value_func_def}
V_k (x) := - \log \psi_k (x) , \qquad x \in \bbR^n 
\end{equation}
serves as a value function of Problem~\eqref{prob:unconstrained}.
The following assumption is used to ensure that the minimum value of Problem~\eqref{prob:unconstrained} is finite.
\begin{assumption}[Finite minimum of Problem~\eqref{prob:unconstrained}]\label{ass:finite_value_func}
    The value function $ V_0 $ in \eqref{eq:value_func_def} and the initial distribution $ \mu_0 $ satisfy
    \begin{equation}
        \int_{\bbR^n} V_0 (x) \mu_0 (\rmd x) < \infty .
    \end{equation}
\end{assumption}
\vspace{0.2cm}

Define the Markov policy on $ \bbU $:
\begin{equation}\label{eq:pibar_def}
\begin{aligned}
    \bar{\pi}_k (\rmd u \mid x) &:= \frac{\psi_{k+1} (f_k(x,u))}{\int_\bbU \psi_{k+1}(f_k(x,v)) R_k (\rmd v)} R_k (\rmd u) .
\end{aligned}
\end{equation}
Denote by $ \bfP^\pi $ the state-input distribution, that is, the joint distribution of the initial state $ x_0 $ and the input sequence $ \{u_k\}_{k=0}^{N-1} $ induced by a policy $ \pi $ with $ x_0 \sim \mu_0 $:
\begin{equation}\label{eq:Ppi_def}
    \bfP^\pi (\rmd x_0, \rmd u_{0:\ft-1}) := \mu_0 (\rmd x_0) \prod_{k=0}^{\ft-1} \pi_k (\rmd u_k \mid h_k ) .
\end{equation}
For the policies $ R= \{R_k\} $ and $ \bar{\pi} = \{\bar{\pi}_k\} $, define
\begin{align}
    \bfR  := \bfP^{R}, \qquad \jref  := \bfP^{\bar{\pi}} . \label{eq:barP_def}
\end{align}

The following result shows that the MaxEnt control problem can be reformulated, up to a policy-independent term, as the minimization of the KL divergence between the state-input distributions $ \bfP^\pi $ and $ \jref $.
The proof is given in Appendix~\ref{app:performance_decompose}.
\begin{proposition}[KL reformulation of MaxEnt control]\label{prop:performance_decompose}
    Under Assumptions~\ref{ass:cost} and \ref{ass:finite_value_func}, any causal policy $ \pi $ supported on $ \bbU $ satisfies
    \begin{equation}\label{eq:performance_decompose}
        J(\pi) = \int_{\bbR^n} V_0 (x) \mu_0 (\rmd x) + \kl{\bfP^\pi}{\jref} .
    \end{equation}
    Moreover, $ \jref $ is the unique optimal state-input distribution for Problem~\eqref{prob:unconstrained}, and $ \bar{\pi} $ in \eqref{eq:pibar_def} is the corresponding optimal policy.
\end{proposition}

Proposition~\ref{prop:performance_decompose} also shows that the minimum value of the terminal-unconstrained problem~\eqref{prob:unconstrained} is $ \int V_0 \rmd \mu_0 $.
In particular, when the initial state is fixed at $ x_0 = x $, the minimum value is $ V_0 (x) $, which clarifies that $ V_0 $ is the value function for \eqref{prob:unconstrained}.

\subsection{Equivalence between MaxEnt distribution steering and \schr~bridges}

Building on the KL reformulation in Proposition~\ref{prop:performance_decompose}, we show that the MaxEnt distribution steering problem is equivalent to an SB problem.

For $ \Omega := \bbR^n \times \bbU^{\ft} $, the deterministic system~\eqref{eq:system} uniquely determines the state trajectory $ x_{0:\ft} $ from each $ \omega = (x_0, u_{0:\ft-1}) \in \Omega$. Thus, $ x_k $ can be regarded as the resulting Borel measurable map $ x_k : \Omega \rightarrow \bbR^n $.  
Define
\begin{equation*}
    T : \Omega \ni \omega \mapsto x_{0:\ft} (\omega) \in (\bbR^n)^{\ft+1} .
\end{equation*}
To relate the KL reformulation to an SB, we transfer the reference state-input distribution $ \jref $ to the state-trajectory space. Pushing $ \jref $ forward through $ T $, define the reference path measure:
\begin{equation}
    \bar{\bbQ}_{0:\ft}  := T_{\#} \jref  . \label{eq:Qbar_def}
\end{equation}
This is the distribution of $ x_{0:\ft} $ driven by the policy $ \bar{\pi} $.
We analyze the SB with the reference \(\bar{\mathbb Q}_{0:\ft}\) using the standard static-to-dynamic reduction~\cite{Leonard2014}. 
We first solve the associated static SB and then reconstruct the dynamic SB from its optimal endpoint coupling and the reference conditional distribution of the intermediate states.

To this end, let $ \bar{\bbQ}_{0,\ft} := (x_0,x_\ft)_\# \jref $ denote the joint distribution of $ (x_0,x_\ft) $ under $ \jref $.
Denote by \(\Gamma(\mu_0,\mu_N)\) the set of couplings of \(\mu_0\) and \(\mu_N\), i.e.,
\begin{align*}
    \Gamma(\mu_0,\mu_\ft) &:= \bigl\{ \gamma \in \calP((\bbR^n)^2) : \gamma(S_0,\bbR^n) = \mu_0(S_0), \\
    &\qquad  \gamma(\bbR^n, S_\ft) = \mu_\ft (S_\ft), \ \forall S_0,S_N \in \scrB(\bbR^n)  \bigr\} .
\end{align*}
Then, we consider the following static SB problem with the reference endpoint distribution \(\bar{\mathbb Q}_{0,N}\).

\begin{problem}[Static SB]\label{prob:static_SB}
    Given a reference measure ${\mathbb{Q}}_{0,N} \in \calP((\bbR^{n})^{2})$ and marginals $ \mu_0, \mu_\ft \in \calP(\bbR^n) $, find a coupling $\gamma \in \calP((\bbR^{n})^{2}) $ that solves
\begin{align} 
    \underset{\gamma \in \Gamma (\mu_0,\mu_\ft)}{\textrm{\rm minimize}} \quad D_{\textrm{\rm KL}}\! \left(\gamma \middle\| {\mathbb{Q}}_{0,N}\right) . \label{eq:static_SB_objective}
\end{align}
\end{problem}
\vspace{0.2cm}

For the static SB problem to admit a solution with finite KL divergence, there must exist at least one coupling of \(\mu_0\) and \(\mu_N\) having finite KL divergence from the reference endpoint distribution. Therefore, we make the following assumption.
\begin{assumption}[Finite-KL coupling]\label{ass:finite_kl}
    There exists a coupling $ \gamma \in \Gamma(\mu_0,\mu_\ft) $ such that
    \begin{equation}\label{eq:finite_kl}
        \kl{\gamma}{\bar{\bbQ}_{0,\ft}} < \infty .
    \end{equation}
\end{assumption}
\vspace{0.2cm}

Under this assumption, the static SB problem has a finite minimum and a unique optimal coupling by a standard existence and uniqueness result for SBs, e.g., in \cite[Theorem~1.10]{Nutz2021introduction}.
\begin{proposition}[Existence and uniqueness]\label{prop:static_SB}
    Under Assumption~\ref{ass:finite_kl}, the static SB problem (Problem~\ref{prob:static_SB}) with $ \bbQ_{0,\ft} = \bar{\bbQ}_{0,\ft} $ has a unique minimizer $ \gamma^* $ and a finite minimum. 
\end{proposition}

By the above result, $ \gamma^* \ll \bar{\bbQ}_{0,\ft} $, and consequently, the following Radon--Nikodym derivative exists:
\begin{equation}\label{eq:G_def}
    G := \frac{\rmd \gamma^*}{\rmd \bar{\bbQ}_{0,\ft}} .
\end{equation}
The optimal coupling \(\gamma^*\) specifies how the endpoint distribution of the reference process is modified to satisfy the prescribed marginals. To recover the entire state trajectory, we retain the conditional distribution of the intermediate states under the reference path measure \(\bar{\mathbb Q}_{0:\ft}\), given the two endpoints.
Denote by $ \bar{\bbQ}_{1:\ft-1\mid 0,\ft}^{\x_0,\x_\ft} $ the conditional distribution of $ \{x_k\}_{k=1}^{\ft-1} $ given $ (x_0,x_\ft) = (\x_0, \x_\ft) $ under $ \bar{\bbQ}_{0:\ft} $, i.e., 
\begin{align}
&\bar{\bbQ}_{1:\ft-1\mid 0,\ft}^{\x_0,\x_\ft} (\rmd x_1,\ldots,\rmd x_{\ft-1})
\nonumber\\
&:=
\bar{\bbQ}_{0:\ft}
( \rmd x_1,\ldots,\rmd x_{\ft-1}
\mid x_0=\x_0,\ x_\ft=\x_\ft
).
\end{align}
The dynamic SB can be reconstructed from the static SB and $ \bar{\bbQ}_{1:\ft-1\mid 0,\ft}^{\x_0,\x_\ft} $ as follows.
\begin{lemma}[Static-to-dynamic reduction of SB]\label{lem:static-to-dynamic}
    Under Assumption~\ref{ass:finite_kl}, the dynamic SB problem (Problem~\ref{prob:SB}) with $ \bbQ_{0:\ft} = \bar{\bbQ}_{0:\ft} $ has the unique minimizer
    \begin{equation}\label{eq:PSB_def}
    \bbP_{0:\ft}^{\rm SB} (\rmd x_{0:\ft}) := \bar{\bbQ}_{1:\ft-1 | 0,\ft}^{x_0,x_\ft} (\rmd x_{1:\ft-1}) \gamma^* (\rmd x_0, \rmd x_\ft) .
    \end{equation}
    Moreover, we have
    \begin{equation}\label{eq:PSB_Qbar}
        \frac{\rmd\bbP^{\rm SB}_{0:\ft}}
            {\rmd\bar{\bbQ}_{0:\ft}} (x_{0:\ft})
        =G(x_0,x_\ft), \quad
\bar{\bbQ}_{0:\ft}\text{-a.e.} 
    \end{equation}
\end{lemma}
\vspace{0.2cm}
\begin{proof}
Since the static SB problem admits the unique minimizer, the dynamic SB problem (Problem~\ref{prob:SB}) has the unique minimizer~$ \bbP_{0:\ft}^{\rm SB} $ given by \eqref{eq:PSB_def} \cite[Proposition~2.3]{Leonard2014}.

To prove \eqref{eq:PSB_Qbar}, recall from \eqref{eq:G_def} that
\[
\gamma^*(\rmd x_0,\rmd x_N)
=
G(x_0,x_N)\bar{\bbQ}_{0,N}(\rmd x_0,\rmd x_N).
\]
Substituting this into \eqref{eq:PSB_def}, we obtain
\[
\begin{aligned}
&\bbP^{\mathrm{SB}}_{0:\ft}(\rmd x_{0:\ft})\\
&=G(x_0,x_N)
\bar{\bbQ}^{x_0,x_N}_{1:\ft-1 | 0,N}(\rmd x_{1:\ft-1})
\bar{\bbQ}_{0,N}(\rmd x_0,\rmd x_N)\\
&=
G(x_0,x_N)\bar{\bbQ}_{0:\ft}(\rmd x_{0:\ft}) ,
\end{aligned}
\]
which shows \eqref{eq:PSB_Qbar}. This completes the proof.
\end{proof}

Lemma~\ref{lem:static-to-dynamic} shows that the dynamic SB is obtained by weighting the reference path measure $ \bar{\bbQ}_{0:\ft} $ by \(G(x_0,x_N)\). Since \(\bar{\mathbb Q}_{0:N}=T_\# \jref \), we lift this weighting to the state-input space $ \Omega $ through the map $ T $ and define
\begin{equation}\label{eq:Pstar_def}
    \bfP^* (\rmd \omega) := G(x_0(\omega), x_\ft(\omega)) \jref (\rmd \omega) .
\end{equation}
The next result shows that this weighted state--input distribution induces the dynamic SB.
\begin{lemma}[State-input realization of the SB]\label{lem:lifted}
    Under Assumption~\ref{ass:finite_kl}, $ \bfP^* $ satisfies
    \begin{equation}\label{eq:Pstar_PSB}
        T_{\#} \bfP^* = \bbP_{0:\ft}^{\rm SB} .
    \end{equation}
    Moreover, there exists a causal policy $ \pi $ supported on $ \bbU $ that generates $ \bfP^* $.
\end{lemma}

\begin{proof}
    For any bounded Borel function $ \varphi : (\bbR^n)^{\ft+1} \rightarrow \bbR $, \eqref{eq:Qbar_def}, \eqref{eq:PSB_Qbar}, and the definition~\eqref{eq:Pstar_def} give
    \begin{align*}
        \int \varphi \rmd (T_{\#} \bfP^*) &= \int \varphi(T(\omega)) \bfP^* (\rmd \omega) \\
        &= \int \varphi (T(\omega)) G(x_0 (\omega), x_\ft (\omega)) \jref (\rmd \omega) \\
        &= \int \varphi (x_{0:\ft}) G(x_0,x_\ft) \bar{\bbQ}_{0:\ft} (\rmd x_{0:\ft})  \\
        &= \int \varphi \rmd \bbP_{0:\ft}^{\rm SB} ,
    \end{align*}
    which proves \eqref{eq:Pstar_PSB}. The existence of a policy $ \pi $ that generates $ \bfP^* $ is guaranteed by the disintegration theorem~\cite[Theorem~8.5]{Kallenberg1997foundations}.
    Since $ \bfP^* $ is supported on $ \bbR^n \times \bbU^{\ft} $, the resulting policy is supported on $ \bbU $.
    This completes the proof.
\end{proof}

Denote by $ \bbP_{k}^\pi \in \calP(\bbR^n) $ the distribution of $ x_k $ under a policy $ \pi $.
The next result shows that a causal policy can be Markovized without changing the state marginals or increasing the objective.
The proof is deferred to Appendix~\ref{app:markov}.
\begin{proposition}[Markovization]\label{prop:markov}
    Let $ \pi $ be any causal policy supported on $ \bbU $ and define the Markov policy
    \begin{equation}
        (\sfM \pi)_k ( U \mid x ) := \bfP^\pi ( u_k \in U \mid x_k = x ) , \quad U \in \scrB(\bbU) .
    \end{equation} 
    Then, for any $ k \in \bbra{0,\ft} $, it holds that
    \begin{equation}\label{eq:state_law_equal}
        \bbP_{k}^{\sfM \pi} = \bbP_{k}^\pi . 
    \end{equation}
    Moreover, if $J(\pi)<\infty$, we have
    \begin{equation}\label{eq:J_pi_Mpi}
        J(\pi) = J(\sfM \pi) + \sum_{k=0}^{\ft-1} \bbE^{\pi} [ \kl{\pi_k (\cdot \mid h_k)}{ (\sfM \pi)_k (\cdot \mid x_k) } ] .
    \end{equation}
\end{proposition}

Consequently, $ J(\sfM \pi) \le J(\pi) $, and the Markovized policy satisfies the same endpoint constraints whenever $ \pi $ does. 
The equality $ J(\sfM \pi) = J(\pi) $ holds if and only if
\begin{equation}\label{eq:markov_opt}
    \pi_k (\cdot \mid h_k) = (\sfM \pi)_k (\cdot \mid x_k) \quad \bfP^\pi\text{-a.s.}
\end{equation}
for any $ k \in \bbra{0,\ft-1} $. Therefore, the infimum of $ J $ over all causal policies is equal to its infimum over Markov policies, and every optimal causal policy with finite objective value is Markov almost surely in the sense of \eqref{eq:markov_opt}.

We are now ready to combine Propositions~\ref{prop:performance_decompose}--\ref{prop:markov} and Lemmas~\ref{lem:static-to-dynamic} and \ref{lem:lifted}.
The following theorem establishes the equivalence between the MaxEnt distribution steering problem and the SB problem with the reference \(\bar{\mathbb Q}_{0:\ft}\).
\begin{theorem}[MaxEnt--SB equivalence]\label{thm:maxent_solves_SB}
    Suppose that Assumptions~\ref{ass:cost}--\ref{ass:finite_kl} hold. 
    Then, $ \bfP^* $ in \eqref{eq:Pstar_def} is the unique optimal state-input distribution for Problem~\ref{prob:MaxEnt}, and there exists a Markov policy $ \pi^* $ supported on $ \bbU $ that generates $ \bfP^* $.
    Consequently, $ \pi^* $ is an optimal policy of Problem~\ref{prob:MaxEnt}.
    Let $ \bbP_{0:\ft}^* $ be the distribution of $ x_{0:\ft} $ under the policy $ \pi^* $. Then, $ \bbP_{0:\ft}^* $ is the unique optimal solution of Problem~\ref{prob:SB} with $ \bbQ_{0:\ft} = \bar{\bbQ}_{0:\ft} $ in \eqref{eq:Qbar_def}, i.e.,
    \begin{equation}
    \bbP_{0:\ft}^* = \bbP_{0:\ft}^{\rm SB} .
    \end{equation}
    Moreover, we have
    \begin{equation}\label{eq:SB_finite}
    \kl{\bbP_{0:\ft}^*}{\bar{\bbQ}_{0:\ft}} = \kl{\bfP^*}{\jref} < \infty.
    \end{equation}
\end{theorem}
\vspace{0.2cm}

\begin{proof}
    For any feasible causal policy $ \pi $ of Problem~\ref{prob:MaxEnt}, Proposition~\ref{prop:performance_decompose} and the data-processing inequality~\cite[Theorem 9 and Example 2]{Van2014renyi} give
    \begin{align}
        J(\pi) &=  \int V_0 \rmd \mu_0 + \kl{\bfP^\pi}{\jref} \nonumber\\
        &\ge \int V_0 \rmd \mu_0 + \kl{T_\# \bfP^\pi}{\bar{\bbQ}_{0:\ft}} \nonumber\\
        &\ge \int V_0 \rmd \mu_0 + \kl{\bbP_{0:\ft}^{\rm SB}}{\bar{\bbQ}_{0:\ft}} , \label{eq:J_lower_bound}
    \end{align}
    where the last inequality follows from the optimality of $\bbP^{\rm SB}_{0:N}$ for the SB problem.
    Recall that the existence of $\bbP^{\rm SB}_{0:N}$ is guaranteed in Lemma~\ref{lem:static-to-dynamic}.

    Next, we show that the above lower bound is attained by $ \pi $ satisfying $ \bfP^\pi = \bfP^* $.
Since $T(\omega)=(x_0(\omega),\ldots,x_\ft(\omega))$, the definition \eqref{eq:Pstar_def} of $ \bfP^* $ and \eqref{eq:PSB_Qbar} imply
\begin{equation}
    \frac{\rmd\mathbf P^*}{\rmd\jref}
    =
    \left(
    \frac{\rmd\bbP^{\mathrm{SB}}_{0:\ft}}{\rmd\bar{\bbQ}_{0:\ft}}
    \right)\circ T,
    \qquad \jref\text{-a.s.} \label{eq:Pstar_Pbar_PSB_Qbar}
\end{equation}
Moreover, by Lemma~\ref{lem:lifted}, we have
\begin{align}
\kl{\mathbf P^*}{\jref}
&=
\int
\log\!\left(
    \frac{\rmd\bbP^{\mathrm{SB}}_{0:\ft}}
         {\rmd\bar{\bbQ}_{0:N}}\circ T
\right)
\rmd\mathbf P^* \nonumber\\
&=
\int
\log\!\left(
    \frac{\rmd\bbP^{\mathrm{SB}}_{0:\ft}}
         {\rmd\bar{\bbQ}_{0:\ft}}
\right)
\rmd(T_\# \mathbf P^*) \nonumber\\
&=
\kl{\bbP^{\mathrm{SB}}_{0:\ft}}{\bar{\bbQ}_{0:\ft}}. \label{eq:kl_equal}
\end{align}
Hence, the lower bound \eqref{eq:J_lower_bound} is attained with equality
for $ \bfP^\pi =  \mathbf P^*$.

By Lemma~\ref{lem:lifted}, there exists a causal
policy $\pi^\circ $ that is feasible for Problem~\ref{prob:MaxEnt} and induces the \red{state-input} distribution
$\mathbf P^*$.  Applying Proposition~\ref{prop:performance_decompose}
to $\pi^\circ $ and using \eqref{eq:kl_equal} yields
\[
\begin{aligned}
J(\pi^\circ )
&=
\int_{\bbR^n}V_0 \rmd\mu_0
+\kl{\mathbf P^*}{\jref}\\
&=
\int_{\bbR^n}V_0\rmd\mu_0
+\kl{\bbP^{\mathrm{SB}}_{0:N}}{\bar{\bbQ}_{0:N}}.
\end{aligned}
\]
Thus, $\pi^\circ $ attains the lower bound \eqref{eq:J_lower_bound} satisfied by any feasible
causal policy, and is therefore optimal for Problem~\ref{prob:MaxEnt}.  In particular,
$\mathbf P^*$ is an optimal \red{state-input} distribution, and $ \bbP_{0:\ft}^* = T_\# \bfP^*  = \bbP_{0:\ft}^{\rm SB} $.

    It remains to prove the uniqueness of $\bfP^*$.
    Let
\begin{align*}
    \mathcal F
    :=
    \{
        \bfP^\pi \in\mathcal P(\Omega) &: \pi \textrm{ causal}, \\
        &
        (x_0)_\#\bfP^\pi=\mu_0,\;
        (x_\ft)_\#\bfP^\pi =\mu_\ft
    \}.
\end{align*}
Let $\bfP_1, \bfP_2 \in \calF$ with finite objective values, and let
$\theta\in(0,1)$. Define
\[
    \bfP_\theta
    :=
    \theta\bfP_1+(1-\theta)\bfP_2.
\]
Since $\bfP_1$ and $\bfP_2$ have the initial marginal $\mu_0$ and
$(x_\ft)_\#\bfP_1=(x_\ft)_\#\bfP_2=\mu_\ft$, linearity of
pushforwards with respect to measures gives $ (x_0)_\# \bfP_\theta = \mu_0$ and $ (x_N)_\# \bfP_\theta = \mu_N$.
Moreover, sequential disintegration yields stochastic kernels
$\{\pi_k^\theta\}_{k=0}^{\ft-1}$ such that
\[
    \bfP_\theta(\rmd x_0,\rmd u_{0:\ft-1})
    =
    \mu_0(\rmd x_0)
    \prod_{k=0}^{\ft-1}
    \pi_k^\theta(\rmd u_k\mid h_k).
\]
Hence, $\bfP_\theta$ is induced by a causal policy and $\bfP_\theta \in \calF $.
Thus, $ \calF $ is convex.

Suppose that $\bfP_1,\bfP_2 \in \calF$ are distinct optimal
state-input distributions for Problem~\ref{prob:MaxEnt}.
By Proposition~\ref{prop:performance_decompose}, they both minimize
$\kl{\cdot}{\jref}$ over the feasible set, and their
KL divergences are finite.
Since $\kl{\cdot}{\jref}$ is strictly convex on its
finite-valued domain, for $ \theta \in (0,1) $,
\[
    D_{\textrm{KL}}(\bfP_\theta\Vert\jref)
    <
    \theta D_{\textrm{KL}}(\bfP_1\Vert\jref)
    +(1-\theta)D_{\textrm{KL}}(\bfP_2\Vert\jref),
\]
which contradicts the optimality of $\bfP_1$ and $\bfP_2$.
Therefore, the optimal state-input distribution is unique,
and hence it must be $\bfP^*$.

Finally, by Proposition~\ref{prop:markov}, the Markovized policy $ \pi^* := \sfM \pi^\circ $ satisfies the endpoint constraints~\eqref{eq:cov_constraint} and cannot increase the objective $ J $. Therefore, $ \pi^* $ is also optimal.
By the uniqueness of $ \bfP^* $, we have $ \bfP^{\pi^*} = \bfP^* $.
The finiteness of the SB objective~\eqref{eq:SB_finite} follows from Proposition~\ref{prop:static_SB} and \eqref{eq:kl_equal}. 
This completes the proof.
\end{proof}

Theorem~\ref{thm:maxent_solves_SB} assumes the existence of a finite-KL coupling (Assumption~\ref{ass:finite_kl}). The next corollary shows that Assumption~\ref{ass:finite_kl} is not merely a sufficient condition, but is equivalent to the existence of a feasible causal policy with finite objective value.

\begin{corollary}[Finite-cost feasibility]\label{cor:finite_cost}
    Under Assumptions~\ref{ass:cost} and \ref{ass:finite_value_func}, the following statements are equivalent:
    \begin{itemize}
        \item[(a)] there exists a causal policy $ \pi $ supported on $ \bbU $ satisfying the endpoint constraints~\eqref{eq:cov_constraint} and $ J(\pi) < \infty $;
        \item[(b)] Assumption~\ref{ass:finite_kl} holds. 
    \end{itemize}
\end{corollary}

\begin{proof}
    If (a) holds, Proposition~\ref{prop:performance_decompose} gives $ \kl{\bfP^\pi}{\jref} < \infty $. Thus, by the data-processing inequality, we have
    \[
        \kl{(x_0,x_\ft)_\# \bfP^\pi}{\bar{\bbQ}_{0,\ft}} < \infty .
    \]
    Since $ (x_0,x_\ft)_\# \bfP^\pi \in \Gamma(\mu_0,\mu_\ft) $, Assumption~\ref{ass:finite_kl} holds.

    Conversely, if (b) holds, then by Theorem~\ref{thm:maxent_solves_SB}, there exists a policy $ \pi^* $ satisfying the endpoint constraints~\eqref{eq:cov_constraint}. Moreover, by Proposition~\ref{prop:performance_decompose} and $ \kl{\bfP^{\pi^*}}{\jref} < \infty $ in \eqref{eq:SB_finite}, we have $ J(\pi^*) < \infty $. This completes the proof.
\end{proof}

When $q_k\equiv0$, we have $\psi_k\equiv1$ and hence
$\bar\pi_k(\cdot\mid x)=R_k$, so that $\bar{\bbQ}_{0:N}$ is the
path measure generated by the state-independent reference inputs
$u_k\sim R_k$.
This parallels the continuous-time SB setting, where the reference
path measure is generated by the uncontrolled noise-driven dynamics~\cite{Pra1990}.
When state costs are present, the terminal-unconstrained MaxEnt
policy $\bar\pi_k$ modifies this reference so as to incorporate the
state costs into the SB reference.

\subsection{Characterization of optimal Markov policies}

The existence of an optimal Markov policy follows from Theorem~\ref{thm:maxent_solves_SB}.
The next proposition gives a direct representation of such a policy.
The proof is deferred to Appendix~\ref{app:opt_policy_condi}.
\begin{proposition}[MaxEnt optimal policy]\label{prop:opt_policy_condi}
    Suppose that Assumptions~\ref{ass:cost}--\ref{ass:finite_kl} hold. 
    For every $ k \in \bbra{0,\ft-1} $, choose a nonnegative Borel function $ \beta_k : \bbR^n \times \bbU \rightarrow [0,\infty] $ such that
    \begin{equation}
        \red{\beta_k (x_k, u_k) = \bbE_{\jref} [ G(x_0,x_\ft) \mid x_k , u_k ] , \quad \jref\text{-a.s.},}
    \end{equation}
    where $ G $ is defined in \eqref{eq:G_def}.
    Define
    \begin{equation}
        \alpha_k(x) := \int_\bbU \beta_k (x,u) \bar{\pi}_k (\rmd u \mid x) .
    \end{equation}
    Let
    \begin{equation}
        w_k (x,u) :=
        \begin{dcases}
            \frac{\beta_k(x,u)}{\alpha_k (x)} , & 0 < \alpha_k (x) < \infty, \\
            1, & \text{otherwise},
        \end{dcases}
    \end{equation}
    and
    \begin{equation}\label{eq:opt_policy_condi}
        \pi_k^* (\rmd u \mid x) := w_k (x,u) \bar{\pi}_k (\rmd u \mid x).
    \end{equation}
    Then, $\pi^* = \{\pi_k^*\}_{k=0}^{\ft-1}$ is an optimal policy of Problem~\ref{prob:MaxEnt}.
\end{proposition}

The above representation of an optimal policy~\eqref{eq:opt_policy_condi} involves a conditional expectation of the endpoint Radon--Nikodym density $ G $.
When $ G $ admits a multiplicative factorization in the initial and terminal states, this representation can be simplified and computed through a backward recursion. The next proposition gives this factorized form of the optimal policy.
The proof is provided in Appendix~\ref{app:doob_opt_policy}.

\begin{proposition}[Factorized optimal policy]\label{prop:doob_opt_policy}
    Suppose that Assumptions~\ref{ass:cost}--\ref{ass:finite_kl} hold. 
    Assume that there exist finite nonnegative Borel functions $ a,b : \bbR^n \rightarrow \bbR $ such that
    \begin{align}
        G(x,y) = a(x) b(y), \quad \bar{\bbQ}_{0,\ft}\text{-a.e.} \label{eq:endpoint_factorization}
    \end{align}
    and
    \begin{equation}\label{eq:initial-factor-positive}
    0 < a (x) < \infty, \quad  \mu_0\text{-a.e.} 
    \end{equation}
    Define 
    \begin{align}
        b_\ft (x) &:= b(x), \nonumber\\
        b_k (x) &:= \int_\bbU b_{k+1} (f_k(x,u)) \bar{\pi}_k (\rmd u \mid x), \quad k = \ft-1,\ldots,0 . \label{eq:markov-h}
    \end{align}
    Let
    \begin{align*}
        \wtilde{w}_k (x,u) :=
        \begin{dcases}
            \frac{b_{k+1} (f_k(x,u))}{b_k (x)}, & 0 < b_k (x) < \infty, \\
            1, & \text{otherwise} ,
        \end{dcases}
    \end{align*}
    and define
    \begin{equation}
        \pi_k^* (\rmd u \mid x) := \wtilde{w}_k (x,u) \bar{\pi}_k (\rmd u \mid x) .
    \end{equation}
    Then, $ \{\pi_k^*\}_{k=0}^{\ft-1} $ is an optimal policy of Problem~\ref{prob:MaxEnt}.
\end{proposition}

Proposition~\ref{prop:doob_opt_policy} relies on the multiplicative factorization of $ G $.
The standard factorization theorem for static SBs \cite[Theorem~2.1]{Nutz2021introduction} yields the following sufficient condition under which this factorization holds.
The proof is given in Appendix~\ref{app:factorization_sufficient}.
\begin{proposition}[Sufficient condition for factorization]\label{prop:factorization_sufficient}
    Let $ S \in \scrB( \bbR^n) $ satisfy $ \mu_N(S) = 1 $ and suppose that Assumption~\ref{ass:finite_kl} holds.
    Then, $ c_S := \bar{\bbQ}_{0,\ft} (\bbR^n \times S) $ satisfies $ c_S > 0 $.
    Define
    \begin{equation}\label{eq:Q_S}
        \bar{\bbQ}_{0,\ft}^S := \frac{\bfone_{\bbR^n \times S} \bar{\bbQ}_{0,\ft}}{c_S}  .
    \end{equation}
    Assume that
    \begin{equation}\label{eq:mutual_continuity}
        \bar{\bbQ}_{0,\ft}^S \ll \mu_0 \otimes \mu_\ft, \quad \mu_0 \otimes \mu_\ft \ll \bar{\bbQ}_{0,\ft}^S .
    \end{equation}
    Then, there exist finite nonnegative Borel functions $ a,b : \bbR^n \rightarrow \bbR $ satisfying \eqref{eq:endpoint_factorization} and \eqref{eq:initial-factor-positive}.
\end{proposition}

\red{
\subsection{Numerical computation}\label{subsec:numerical}

We describe a finite-state approximation of SBs based on the static-to-dynamic
reduction in Lemma~\ref{lem:static-to-dynamic}.
Let $\{\calC_j\}_{j=1}^{M}$ be cells used to discretize
the state space, with representative points $\rmx_j\in \calC_j$.
Let $P_k^R=(P_{k,ij}^R)\in\bbR^{M\times M}$ be the resulting
row-stochastic approximation of the reference transition, where
$P_{k,ij}^R$ approximates the probability of
$x_{k+1}\in \calC_j$ given $x_k=\rmx_i$ and $u_k\sim R_k$.
We discretize the soft-Bellman equation~\eqref{eq:desire_bellman} and the reference policy \eqref{eq:pibar_2}
to construct the terminal-unconstrained reference process:
\begin{align*}
    \psi_{\ft,i}&=1,\qquad
    \psi_{k,i}
    =
    \Ee^{-q_k(\rmx_i)}
    \sum_j P_{k,ij}^R\psi_{k+1,j}, \\
    \bar P_{k,ij}
    &=
    \Ee^{-q_k(\rmx_i)}P_{k,ij}^R
    \frac{\psi_{k+1,j}}{\psi_{k,i}},
    \qquad k=\ft-1,\ldots,0.
\end{align*}

Let $m_0,m_\ft\in\bbR_{\ge0}^{M}$ be probability mass vectors
approximating $\mu_0,\mu_\ft$ on the grid.
The reference endpoint distribution is represented by
$
    Q:=\operatorname{diag}(m_0)
    \bar P_0\bar P_1\cdots\bar P_{\ft-1}
$,
where $\diag(m_0)$ denotes the diagonal matrix whose
diagonal entries are given by $m_0$.
The finite-state static SB problem is
\[
\begin{aligned}
    \underset{\Gamma\ge0}{\operatorname{minimize}}
    \quad&
    \sum_{i,j}\Gamma_{ij}\log\frac{\Gamma_{ij}}{Q_{ij}}\\
    \text{subject to}\quad&
    \Gamma\mathbf 1=m_0,\qquad
    \Gamma^\top\mathbf 1=m_\ft.
\end{aligned}
\]
Let
$
\calI_0:=\{i:m_{0,i}>0\} $ and 
$\calI_N:=\{j:m_{N,j}>0\}$.
Assume that $Q_{ij}>0$ for all $i\in \calI_0$ and $j\in \calI_N$.
Then, applying the Sinkhorn algorithm~\cite{Peyre2019} to the submatrix
$Q_{\calI_0,\calI_N}$ yields the optimal solution in the form
\[
    \Gamma_{\calI_0,\calI_N}^*
    =
    \operatorname{diag}(a)
    Q_{\calI_0,\calI_N}
    \operatorname{diag}(b),
\]
where $a$ and $b$ are chosen so that
$\Gamma_{\calI_0,\calI_N}^*\mathbf 1=m_{0,\calI_0}$ and
$(\Gamma_{\calI_0,\calI_N}^*)^\top\mathbf 1=m_{N,\calI_N}$.
Extend $b$ by zero outside $\calI_N$ to a
vector $\wtilde{b}\in\bbR^M$. 
Then, by the finite-state counterpart of Proposition~\ref{prop:doob_opt_policy}, the optimal transition matrices $ \{P_k^*\}_k $ for the SB and state marginals $ \{m_k\} $ are obtained through
\begin{align*}
    &h_\ft=\wtilde{b},\qquad h_k=\bar P_kh_{k+1}, \\
    &P_{k,ij}^*
    =
    \bar P_{k,ij}\frac{h_{k+1,j}}{h_{k,i}},
    \qquad
    m_{k+1}=(P_k^*)^\top m_k.
\end{align*}
For rows with $h_{k,i}=0$, which carry no bridge mass, set
$P_{k,ij}^* :=\bar P_{k,ij}$.

To obtain a control policy from $\{P_k^*\}$, for each pair $(i,j)$ with
$P_{k,ij}^*>0$, choose a representative input
$u_{k,ij}\in\bbU$ satisfying
$
    f_k(\rmx_i,u_{k,ij})\in \calC_j
$.
Define
\[
    \pi_k(\rmd u\mid \rmx_i)
    :=
    \sum_{j=1}^M
    P_{k,ij}^*\,\delta_{u_{k,ij}}(\rmd u),
\]
where $\delta_u$ denotes the Dirac measure at $u$.
Under the finite-state approximation, this policy induces the
cell-to-cell transition probabilities specified by $P_k^*$.

\begin{figure}[t]
    \centering

    \begin{subfigure}[t]{0.49\linewidth}
        \centering
        \includegraphics[width=\linewidth]{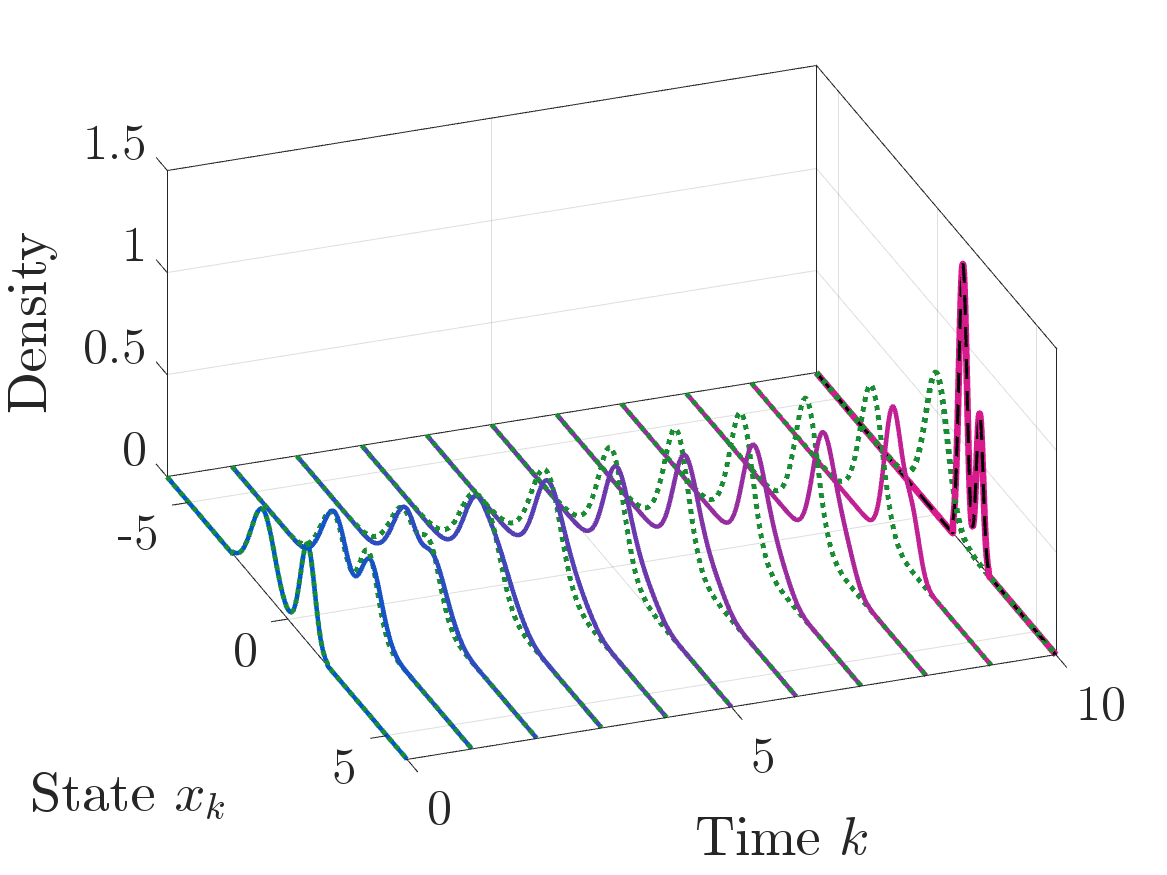}
        \caption{Without state cost}
        \label{fig:eps1}
    \end{subfigure}
    \hfill
    \begin{subfigure}[t]{0.49\linewidth}
        \centering
        \includegraphics[width=\linewidth]{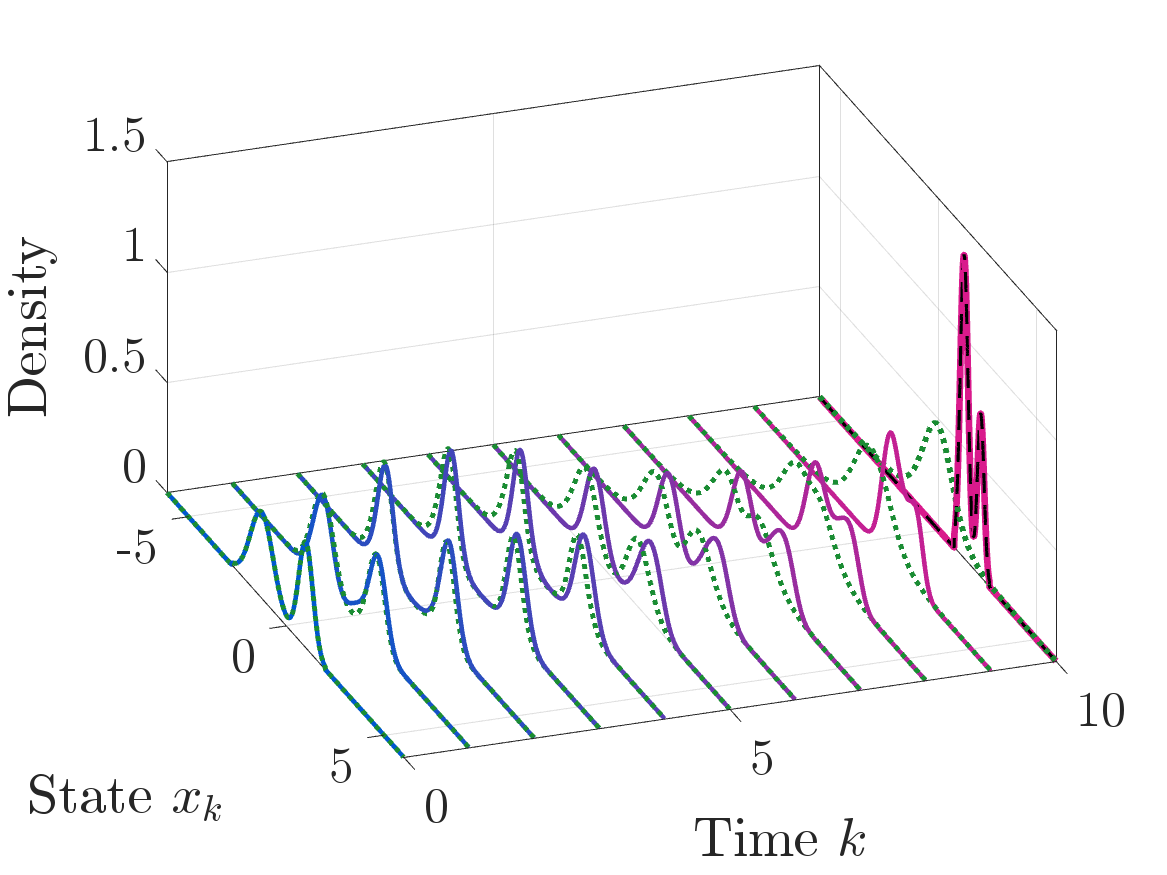}
        \caption{With state cost \eqref{eq:num-state-cost}}
        \label{fig:eps2}
    \end{subfigure}

    \caption{Evolution of the state density for the nonlinear system \eqref{eq:num-system}
with non-Gaussian endpoint distributions \eqref{eq:num-p0} and
\eqref{eq:num-pN}. Solid lines: densities of $x_k$ under the optimal policy; dotted lines: time marginals of the corresponding terminal-unconstrained
MaxEnt reference $\bar{\bbQ}_{0:N}$ in \eqref{eq:Qbar_def}; dashed lines: target density $ \rho_N $.}
    \label{fig:num-snapshots}
\end{figure}

\begin{example}[Nonlinear non-Gaussian steering]\label{ex:numerical}
Consider
\begin{equation}\label{eq:num-system}
    x_{k+1}=x_k - 0.25\sin x_k+u_k,
\end{equation}
with $\bbU=\bbR$, $\ft=10$, and the input cost
$
    \ell_k(u)=\frac{u^2}{2(0.35)^2}
$.
To illustrate the effect of the state cost, set
\begin{equation}
q_k(x)=
\begin{cases}
4\exp\!\left(-\frac{x^2}{2(0.8)^2}\right),&k=2,3,4,\\
0,&\text{otherwise}.
\end{cases}
\label{eq:num-state-cost}
\end{equation}
The initial and target densities are the truncated Gaussian mixtures
\begin{align}
 \rho_0(x)&\propto
 \left[
0.45\Ee^{-\frac{(x+1.25)^2}{2(0.58)^2}}
+
0.55\Ee^{-\frac{(x-1.00)^2}{2(0.36)^2}}
\right]
\bfone_{[-2.8,2.0]}(x),
 \label{eq:num-p0}\\
 \rho_N(x)&\propto
 \left[
 0.65\Ee^{-\frac{(x-1.30)^2}{(2(0.20)^2)}}
 +0.35\Ee^{-\frac{(x-2.15)^2}{(2(0.17)^2)}}
 \right]\bfone_{[0.8,2.6]}(x).
 \label{eq:num-pN}
\end{align}
In this setting, Assumptions~\ref{ass:cost} and \ref{ass:finite_value_func} are satisfied.
Moreover, in Subsection~\ref{subsec:control_affine}, we will see that Assumption~\ref{ass:finite_kl} also holds. Hence, Theorem~\ref{thm:maxent_solves_SB} ensures the existence of an optimal policy.

We discretize $[-7,7]$ using $M=1501$ equally spaced grid points $ \{\rmx_i\} $ and set
\[
    u_{k,ij}:=\rmx_j-\rmx_i + 0.25\sin \rmx_i,\quad
    P_{k,ij}^R\propto
    \exp\!\left(-\ell_k (u_{k,ij})\right),
\]
with each row normalized to sum to one.
Then, we apply the above Sinkhorn procedure.

Fig.~\ref{fig:num-snapshots} compares the optimal density evolution
with the corresponding terminal-unconstrained MaxEnt reference $ \bar{\bbQ}_{0:N} $.
To reduce the state costs incurred at $k=2,3,4$, the
terminal-unconstrained reference with \eqref{eq:num-state-cost}
shifts probability mass away from the origin during the
first half of the horizon.
Consequently, under the optimal controller with the state cost, the state density also avoids the region around the origin.
After the penalized stages, the controlled density
evolves towards the target distribution.
\hfill $ \diamondsuit $
\end{example}

}

\section{Killing and Feynman--Kac Interpretation of the State Cost}\label{sec:killing}

\begin{figure}[t]
        \centering
        \includegraphics[width=0.7\linewidth]{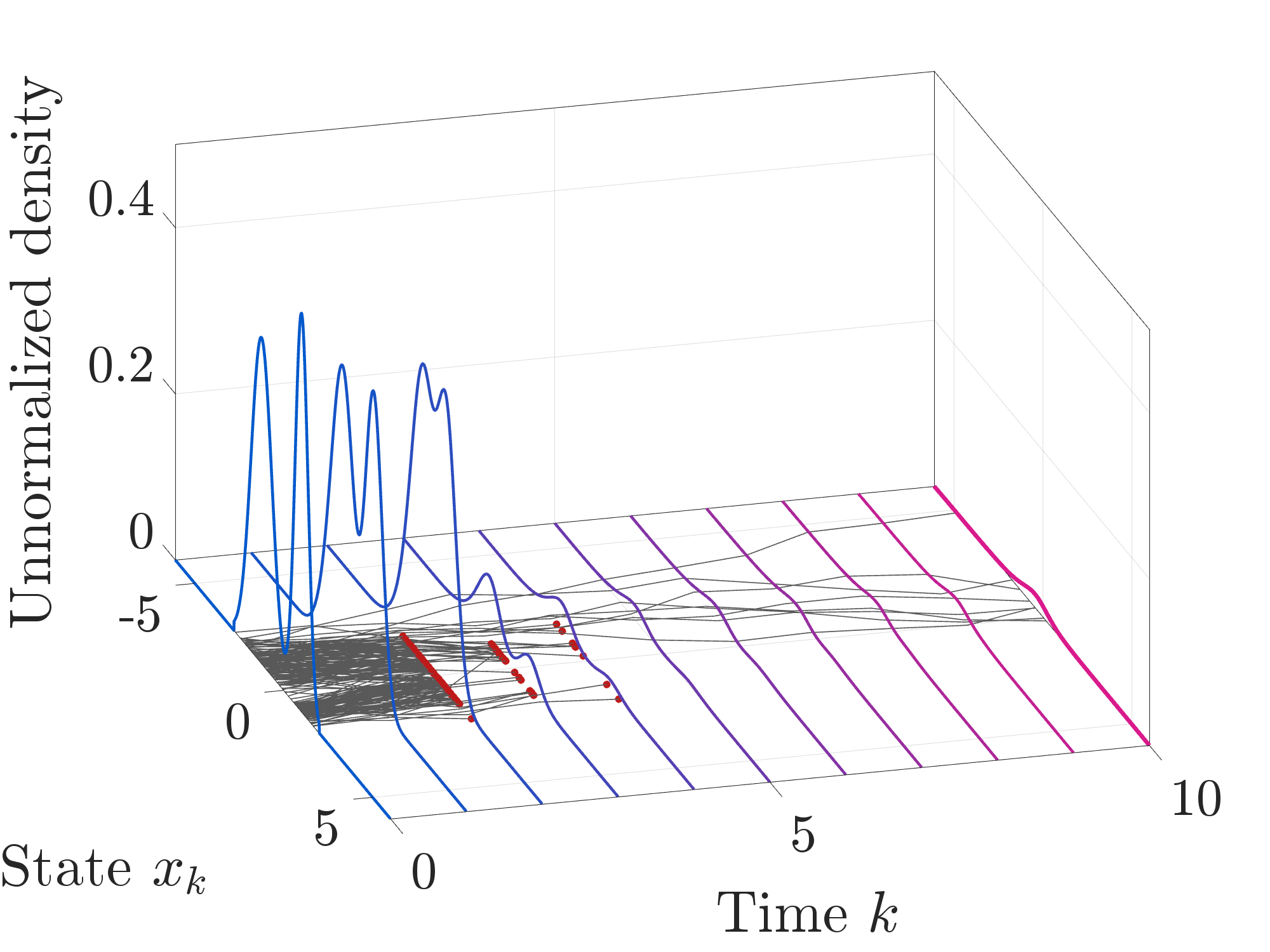}
\caption{
Evolution of the unnormalized killed state density and $200$
sample paths for system~\eqref{eq:num-system}
driven by $u_k\sim R_k$, under the input cost $\ell_k$
and survival probability $\Ee^{-q_k(x)}$
specified in Example~\ref{ex:numerical}.
Colored solid curves: unnormalized state densities;
gray curves on the bottom plane: sample paths;
red dots: killing events.
}
\label{fig:killed-process}
\end{figure}

We now reinterpret the state cost as a killing mechanism. This provides an alternative reference measure that leads to the same SB as the path measure $ \bar{\bbQ}_{0:\ft} $ induced by the terminal-unconstrained optimal policy. It also clarifies the relationship among the killed reference, its normalized Feynman--Kac measure, and the terminal-unconstrained optimal path measure.

We introduce killing into the baseline process generated by $ \bfR $.
Define the one-step survival probability at state $ x $:
\begin{equation}\label{eq:killing_probability}
    G_k (x) := \Ee^{-q_k (x)} \in (0,1], \quad k\in \bbra{0,\ft-1} .
\end{equation}
The killed process can be described as follows.
If a particle is alive at time $ k $ and is located at \(x_k=x\), it survives that step with
probability \(G_k(x)\), and is killed with probability \(1-G_k(x)\).
Conditional on survival, it draws $ u_k \sim R_k $ and moves according to $ x_{k+1} = f_k (x_k, u_k) $.

\begin{example}[Continued]
\red{Fig.~\ref{fig:killed-process} illustrates the killed process
for system~\eqref{eq:num-system}, driven by $u_k\sim R_k$
with $R_k$ determined by the input cost $\ell_k$ in
Example~\ref{ex:numerical}, and subject to the survival
probability $\Ee^{-q_k(x_k)}$ with $q_k$ given by
\eqref{eq:num-state-cost}.
The densities show the unnormalized distributions
of particles surviving up to each time, whose total mass decreases
as killing occurs.
The sample paths on the bottom plane terminate at their killing times.
The probability of survival through the horizon, denoted by $ Z_{\rm kill} $ later, equals the total mass of the killed density at the final time and is $ 0.0376 $.
}
\hfill $ \diamondsuit $
\end{example}

For \red{an initial state and an input} trajectory
$
\omega=(x_0,u_{0:\ft-1})
$, the probability of surviving all
stages along the trajectory is
$
\prod_{k=0}^{N-1}G_k(x_k(\omega))
$.
We therefore define the killed \red{state-input} reference by weighting
the \red{state-input} distribution \(\mathbf R\) with this cumulative
survival probability:
\begin{align}
    \mathbf K(\rmd\omega)&:=
      \left\{\prod_{k=0}^{N-1}G_k(x_k(\omega))\right\}
      \bfR(\rmd\omega) \nonumber\\
    &=
      \exp\!\left(
        -\sum_{k=0}^{N-1}q_k(x_k(\omega))
      \right)\bfR(\rmd\omega)  .
    \label{eq:killed-joint-measure}
\end{align}
The finite measure $ \bfK $ is generally not normalized. Its total mass is
the probability of survival over the horizon.

For a nonzero finite measure $ M $ and a probability measure $ P $ on the same measurable space, define the KL divergence between $ P $ and $ M $ by
\begin{equation}
    \klg{P}{M} :=
    \begin{dcases}
        \int \log \left( \frac{\rmd P}{\rmd M} \right) \rmd P , & P \ll M,\\
        + \infty , & P \not\ll M .
    \end{dcases}
\end{equation}
For the normalized measure $ Q := M / \int \rmd M $, it holds that
\begin{equation}\label{eq:klg_kl}
    \klg{P}{M} = \kl{P}{Q} - \log \left(\int \rmd M \right) .
\end{equation}

The following result shows that the cumulative state cost can be absorbed into the killed reference measure $ \bfK $.

\begin{proposition}[Killed-reference KL representation]\label{prop:maxent_killed}
    For any causal policy $ \pi $, it holds that
    \begin{equation}\label{eq:maxent_obj_killing}
        J(\pi) = \klg{\bfP^\pi}{\bfK} .
    \end{equation}
\end{proposition}
\vspace{0.2cm}

\begin{proof}
    If $ \bfP^\pi \not\ll \bfR $, then $ \bfP^\pi \not\ll \bfK $, because $ \bfK $ and $ \bfR $ have the same null sets. In this case, both sides of \eqref{eq:maxent_obj_killing} are $ + \infty $.

    In what follows, we suppose that $ \bfP^\pi \ll \bfR $. Since
    \[
        \frac{\rmd \bfK}{\rmd \bfR} = \exp \! \left( - \sum_{k=0}^{\ft-1} q_k (x_k) \right) ,
    \]
    we have
    \[
        \frac{\rmd \bfP^\pi}{\rmd \bfK} = \frac{\rmd \bfP^\pi}{\rmd \bfR} \exp \! \left( \sum_{k=0}^{\ft-1} q_k (x_k) \right) .
    \]
    Taking the logarithm and integrating with respect to $ \bfP^\pi $, we obtain
    \begin{equation*}
        \klg{\bfP^\pi}{\bfK} = \kl{\bfP^\pi}{\bfR} + \bbE^\pi \! \left[ \sum_{k=0}^{\ft-1} q_k (x_k) \right] .
    \end{equation*}
    Substituting \eqref{eq:Ppi_R_kl} into this equation yields the desired result~\eqref{eq:maxent_obj_killing}.
\end{proof}

This proposition provides an alternative KL representation of the MaxEnt control objective, in which the state cost is incorporated into $ \bfR $ through the killing weights. Unlike the reference probability measure $ \jref $ in Proposition~\ref{prop:performance_decompose}, $ \bfK $ is generally an unnormalized finite measure. 

To relate this finite reference measure to a standard probability reference for the SB problem, we next normalize $ \bfK $.
Define the normalized Feynman--Kac measure $ \bfQ^{\FK} $ \cite{DelMoral2004} on the space of initial states and input trajectories by
\begin{equation}
    \bfQ^{\FK} := \frac{1}{Z_{\kil}} \bfK , \qquad Z_\kil := \bfK (\Omega) .
\end{equation}
We also define the path measures corresponding to $ \bfK $ and $ \bfQ^{\rm FK} $ by
\begin{equation}\label{eq:kill_path_def}
    \bbK_{0:\ft} := T_\# \bfK, \qquad \bbQ_{0:\ft}^\FK := T_\# \bfQ^\FK .
\end{equation}

To characterize the Feynman--Kac measure, we consider the probability of surviving from an intermediate state $ x $ at time $ k $ until the terminal time.
Let 
\[
    s_k(x)
    :=
    \bbE_{k,x}^{R} \Biggl[
        \prod_{j=k}^{N-1}G_j(x_j)
    \Biggr], \quad k\in \bbra{0,\ft},
\]
where \(\bbE_{k,x}^{R}\) denotes expectation for the state process
starting from \(x_k=x\) and driven by \(u_j\sim R_j\),
\(j\in \bbra{k,N-1}\).
Then, the survival probabilities $ \{s_k\} $ satisfy the same backward recursion as the desirability functions $ \{\psi_k\} $ in \eqref{eq:desire_bellman}. The next lemma uses this observation to relate the normalized Feynman--Kac measure to the terminal-unconstrained optimal \red{state-input} distribution $ \jref $.
\begin{lemma}[Feynman--Kac reference characterization]\label{lem:s_psi}
    Under Assumption~\ref{ass:cost}, it holds that $ s_k = \psi_k $ for any $ k \in \bbra{0,\ft} $. In particular,
    \begin{equation}
        Z_\kil = \int_{\bbR^n} \psi_0 (x) \mu_0 (\rmd x) \in (0,1] .
    \end{equation}
    Moreover, it holds that
    \begin{align}
        \bfQ^\FK (\rmd x_0, \rmd u_{0:\ft-1}) &= \mu_0^\FK (\rmd x_0) \prod_{k=0}^{\ft-1} \bar{\pi}_k (\rmd u_k \mid x_k), \label{eq:Q_FK}  \\
        \frac{\rmd \jref}{\rmd \bfQ^{\FK}} &= \frac{Z_\kil}{\psi_0 (x_0)} , \label{dPbar_dQFK}
    \end{align}
    where
    \begin{equation}
        \mu_0^\FK (\rmd x) := \frac{\psi_0 (x)}{Z_\kil} \mu_0 (\rmd x) .
    \end{equation}
\end{lemma}
\vspace{0.2cm}

\begin{proof}
    First, we calculate the survival probability $ s_k (x) $. Starting from the state $ x $ at time $ k $, the particle survives with probability $ G_k (x) = \Ee^{-q_k (x)} $, draws $ u_k \sim R_k $, moves to $ f_k (x, u_k) $, and then survives from time $ k+1 $ with probability $ s_{k+1} (f_k (x,u_k)) $. Hence, $ s_k $ satisfies $ s_\ft (x) \equiv 1 $ and
    \begin{align*}
        s_k (x) &= \Ee^{-q_k(x)} \int_\bbU s_{k+1} (f_k (x,u)) R_k (\rmd u) .
    \end{align*}
    This is exactly \eqref{eq:desire_boundary} and \eqref{eq:desire_bellman} defining $ \psi_k $, and thus, $ s_k = \psi_k $. Since $ Z_\kil := \bfK(\Omega) $ is the probability of survival over the horizon, we obtain
    \[
        Z_\kil = \int_{\bbR^n} s_0 (x) \mu_0 (\rmd x) = \int_{\bbR^n} \psi_0 (x) \mu_0 (\rmd x) .
    \]

    By \eqref{eq:pibar_2}, we have
    \begin{align*}
        \prod_{k=0}^{\ft-1} \bar{\pi}_k (\rmd u_k \mid x_k) = \frac{\exp ( - \sum_{k=0}^{\ft-1} q_k (x_k) )}{\psi_0 (x_0)} \prod_{k=0}^{\ft-1} R_k (\rmd u_k) .
    \end{align*}
    Therefore, it holds that
    \begin{align*}
        &\exp \Biggl( - \sum_{k=0}^{\ft-1} q_k (x_k) \Biggr) \mu_0 (\rmd x_0) \prod_{k=0}^{\ft-1} R_k (\rmd u_k) \\
        &= \psi_0 (x_0) \mu_0 (\rmd x_0) \prod_{k=0}^{\ft-1} \bar{\pi}_k (\rmd u_k \mid x_k) .
    \end{align*}
    Dividing by $ Z_\kil $ and using \eqref{eq:killed-joint-measure} yields \eqref{eq:Q_FK}. Lastly, by \eqref{eq:Q_FK} and the definition of $ \jref $ in \eqref{eq:barP_def}, we obtain \eqref{dPbar_dQFK}.
\end{proof}

Lemma~\ref{lem:s_psi} shows that $ \bfQ^{\rm FK} $ and $ \jref $ have the same Markov input kernels and differ only in their initial distributions. 
In particular, their Radon--Nikodym derivative depends only on $ x_0 $.
Since the SB problem fixes the initial marginal to $ \mu_0 $, this difference changes the KL objective only by a constant. This leads to the following equivalence.

\begin{theorem}[Feynman--Kac and killed references]\label{thm:killed_FK_equiv}
     Suppose that Assumptions~\ref{ass:cost}--\ref{ass:finite_kl} hold. Let $ \bbP_{0:\ft}^* $ be the path measure under an optimal policy of Problem~\ref{prob:MaxEnt}. 
     Define $ \bbQ_{0:\ft}^{\rm FK} $ and $ \bbK_{0:\ft} $ by \eqref{eq:kill_path_def}.
     Then, $ \bbP_{0:\ft}^* $ is the unique optimal solution of Problem~\ref{prob:SB} with $ \bbQ_{0:\ft} = \bbQ_{0:\ft}^\FK $. Moreover, $ \bbP_{0:\ft}^* $ is the unique optimal solution of the following problem:
     \begin{align} 
    &\underset{\mathbb{P}_{0:N} \in \calP((\bbR^{n})^{N+1})}{\textrm{\rm minimize}} && \quad \calD\! \left(\mathbb{P}_{0:N} \middle\| {\bbK}_{0:N}\right) ,  \label{eq:SB_objective_kill}\\
    &\quad \textrm{\rm ~subject to}  && (T_0)_{\#} \mathbb{P}_{0:N} = \mu_0, \ (T_\ft)_{\#} \mathbb{P}_{0:N} = \mu_\ft , \nonumber
\end{align}
where $ \mu_0, \mu_\ft \in \calP(\bbR^n) $ and $ T_k $ is given in Problem~\ref{prob:SB}.
\end{theorem}

\begin{proof}
    Pushing \eqref{dPbar_dQFK} forward under $ T $, we have
    \begin{equation*}
        \frac{\rmd \bar{\bbQ}_{0:\ft}}{\rmd \bbQ_{0:\ft}^\FK} = \frac{Z_\kil}{\psi_0 (x_0)} .
    \end{equation*}
    Let $ P_{0:\ft} \in \calP((\bbR^{n})^{N+1}) $ be any path measure with $ (T_0)_\# P_{0:\ft} = \mu_0 $. 
    By the above equation, it holds that
    \begin{align}
        &\kl{P_{0:\ft}}{\bbQ_{0:\ft}^\FK} \nonumber\\
        &= \kl{P_{0:\ft}}{\bar{\bbQ}_{0:\ft}} + \bbE_{P_{0:\ft}} \! \left[ \log \frac{\rmd \bar{\bbQ}_{0:\ft}}{\rmd \bbQ_{0:\ft}^\FK} \right] \nonumber\\
        &= \kl{P_{0:\ft}}{\bar{\bbQ}_{0:\ft}} + \log Z_\kil + \int_{\bbR^n} V_0 (x) \mu_0 (\rmd x) .
    \end{align}
    Thus, it follows from \eqref{eq:klg_kl} that the three objectives $ \kl{P_{0:\ft}}{\bar{\bbQ}_{0:\ft}} $, $ \kl{P_{0:\ft}}{\bbQ_{0:\ft}^\FK} $, and $ \klg{P_{0:\ft}}{\bbK_{0:\ft}} $ differ only by constants on the set $ \{ P_{0:\ft} \in \calP((\bbR^n)^{\ft+1}) : (T_0)_{\#} P_{0:N} = \mu_0, \ (T_\ft)_{\#} P_{0:N} = \mu_\ft \} $, and hence have the same minimizer.
    Moreover, by Theorem~\ref{thm:maxent_solves_SB}, $ \bbP_{0:\ft}^* $ is the unique optimal solution of Problem~\ref{prob:SB} with $ \bbQ_{0:\ft} = \bar{\bbQ}_{0:\ft} $. This completes the proof.
\end{proof}

\begin{remark}[Linear-quadratic-Gaussian case]
Section~\ref{sec:reachability} gives reachability-based sufficient conditions for
Assumption~\ref{ass:finite_kl}, used in Theorems~\ref{thm:maxent_solves_SB} and~\ref{thm:killed_FK_equiv}, under compactly supported
endpoint distributions. In the linear-quadratic-Gaussian setting
of~\cite{Ito2022maximum}, however, Assumptions~\ref{ass:cost}--\ref{ass:finite_kl} hold under the reachability condition, even for Gaussian marginals with
unbounded support. Hence, Theorem~\ref{thm:killed_FK_equiv} shows that the quadratic state
cost considered in \cite{Ito2025maximum} admits a killing interpretation.
\hfill $ \diamondsuit $
\end{remark}

\section{Reachability conditions for finite-KL couplings}\label{sec:reachability}

Theorems~\ref{thm:maxent_solves_SB} and \ref{thm:killed_FK_equiv} rely on Assumption~\ref{ass:finite_kl}, which requires the existence of an endpoint coupling having finite KL divergence from the reference $ \bar{\bbQ}_{0,N} $. 
In this section, we derive reachability-based sufficient conditions for this assumption.

\subsection{Nonlinear systems}

For \(x\in\bbR^n\) and
\(\mathbf u=(u_0,\ldots,u_{N-1})\in\bbU^N\), let
\(x_k^{x,\mathbf u}\) denote the trajectory defined by
\begin{equation}\label{eq:state_u}
    x_0^{x,\mathbf u}=x,\qquad
    x_{k+1}^{x,\mathbf u}
    =f_k(x_k^{x,\mathbf u},u_k).
\end{equation}
We identify $ \bfu $ with the stacked vector $ [u_0^\top,\ldots,u_{N-1}^\top]^\top \in \bbR^{mN} $.
To relate the reference endpoint distribution to reachability of the system~\eqref{eq:system}, we work with the endpoint map from input sequences to terminal states, defined by
\[
    \calE_x(\mathbf u):=x_N^{x,\mathbf u}.
\]
The following regularity assumption ensures that this map is differentiable on the interior of the admissible input set and allows us to apply the coarea formula below.
\begin{assumption}[Regular input set and smooth dynamics]\label{ass:reg_input_smooth_dyn}
    The interior of the input set $ \calO := \mathop{\rm int} \bbU $ satisfies
    \begin{equation}
        \calO \neq \varnothing, \quad \mu_L^m (\bbU \setminus \calO) = 0 .
    \end{equation}
    In addition, for each $ k \in \bbra{0,\ft-1} $, $ f_k $ is $ C^1 $ on $ \bbR^n \times \calO $.
\end{assumption}

Set
\[
    \calA_k(x,\mathbf u)
    :=D_xf_k(x_k^{x,\mathbf u},u_k),
    \quad
    \calB_k(x,\mathbf u)
    :=D_\bfu f_k(x_k^{x,\mathbf u},u_k),
\]
where $ D_x f_k $ and $ D_\bfu f_k $ denote the Jacobian matrices of $ f_k $ with respect to its state and input arguments, respectively.
The matrices $\calA_k(x,\bfu)$ and $\calB_k(x,\bfu)$ define the
linearization of the system $ x_{k+1} = f_k(x_k,u_k) $ along the trajectory
generated by $(x,\bfu)$. 
Consider small perturbations $ \delta x_k $ and $ \delta u_k $ of the state and input around this trajectory.
To first order, they satisfy
\[
    \delta x_{k+1}
    =
    \calA_k(x,\bfu)\delta x_k
    +
    \calB_k(x,\bfu)\delta u_k.
\]
Define the state-transition matrix for this linearized system:
\begin{equation*}
    \Phi_{x,\bfu} (k,j) :=
    \begin{dcases}
        \calA_{k-1} (x, \bfu) \cdots \calA_j (x,\bfu), & k > j, \\
        I ,  &  k = j .
    \end{dcases}
\end{equation*}
For a fixed initial state, the variation of the terminal state induced
by perturbations of the input sequence is described by the
following expression:
    \begin{align}\label{eq:reachability_matrix}
        &D_\bfu \calE_x (\bfu) \nonumber\\
        &= \begin{bmatrix} 
            \Phi_{x,\bfu} (\ft,1) \calB_0 (x,\bfu),\ldots,\Phi_{x,\bfu} (\ft,\ft) \calB_{\ft-1} (x,\bfu) 
        \end{bmatrix} .
    \end{align}

Thus, $D_\bfu \calE_x(\bfu)$ plays the role of a finite-horizon
reachability matrix for the linearized system along the trajectory.
In particular, $\rank D_\bfu \calE_x(\bfu)=n$ means that
arbitrary first-order variations of the terminal state can be
generated by perturbing the input sequence around $\bfu$.
Since $ D_\bfu \calE_x (\bfu) \in \bbR^{n\times mN}  $, the rank condition $ \rank D_\bfu \calE_x (\bfu) = n $ can hold only if $ mN \ge n $.

Write
\begin{align*}
    \calJ_x (\bfu) &:= \sqrt{ \det \! \left(D_\bfu \calE_x (\bfu) D_\bfu \calE_x (\bfu)^\top \right) } .
\end{align*}
Now, we are ready to connect reachability of \eqref{eq:system} to Assumption~\ref{ass:finite_kl}. 
\begin{theorem}[Reachability-based finite-KL feasibility]\label{thm:uniform-reachability-finite-entropy}
    Suppose that Assumptions~\ref{ass:cost} and \ref{ass:reg_input_smooth_dyn} hold, $ d := mN \ge n $, and the cost functions $ q_k $ and $ \ell_k $ are continuous for any $ k \in \bbra{0,N-1} $.
    Let $ X, S \subset \bbR^n $ be compact sets such that
    \begin{equation}
        \mu_0 (X) = 1, \quad \mu_\ft (S) = 1.
    \end{equation}
    Assume that $ \mu_\ft $ has a bounded density function $ \rho_\ft $.
    Assume also that there exist a compact set $ C \subset \calO^N $ and a constant $ h_0 > 0 $ such that, for $ \mu_0 $-almost every $ x \in X $ and $ \mu_L^n $-almost every $ y \in S $,
    \begin{equation}\label{eq:Hausdorff_positive}
        \calH^{d-n} \! \left( \calE_x^{-1} (y) \cap C \cap \{\calJ_x > 0\} \right) \ge h_0 ,
    \end{equation}
    where \(\mathcal H^j\) denotes the \(j\)-dimensional Hausdorff measure.
    Then, Assumption~\ref{ass:finite_kl} holds.
    In particular, the coupling $ \gamma^\circ = \mu_0 \otimes \mu_\ft \in \Gamma(\mu_0, \mu_\ft) $ satisfies 
    \[
        \kl{\gamma^\circ}{\bar{\bbQ}_{0,\ft}} < \infty .
    \]
\end{theorem}
\vspace{0.2cm}

\begin{proof}
    For $ x\in X $ and $ \bfu \in \calO^N $, set
    \begin{align}
        Q_x (\bfu) &:= \sum_{k=0}^{\ft-1} q_k (x_k^{x,\bfu}) , \\
        r^{(N)} (\bfu) &:= \prod_{k=0}^{\ft-1} r_k (u_k), \qquad r_k (u) := Z_{\ell,k}^{-1} \Ee^{-\ell_k (u)} .
    \end{align}
    By the compactness of $ X $ and $ C $ and the continuity of $ f_k $, $ q_k $, and $ \ell_k $, there exist constants $ \underline{r} > 0 $, $ \bar{\calJ} \in (0,\infty) $, and $ \bar{q} \in [0,\infty) $ such that for any $ x\in X $ and $ \bfu \in C \cap \{\calJ_x > 0\} $,
    \begin{equation}\label{eq:r_J_Q_bound}
        r^{(N)} (\bfu) \ge \underline{r}, \quad \calJ_x (\bfu) \le \bar{\calJ}, \quad Q_x (\bfu) \le \bar{q} .
    \end{equation}
    Indeed, let $ C_k := {\rm proj} _k (C) $, where $ {\rm proj}_k(\bfu) := u_k $. Since $ C_k $ is compact and $ r_k $ is positive and continuous on $ C_k $, $ \min_{u \in C_k} r_k(u) > 0 $.  
    The finiteness of $ \bar{q} $ and $ \bar{\calJ} $ follows from the fact that the maps $ (x,\bfu) \mapsto Q_x (\bfu) $ and $ (x,\bfu ) \mapsto \calJ_x (\bfu) $ are continuous on the compact set $ X \times C $.

    For each $ x \in \bbR^n $, let $ \jref^x $ be the distribution of the input sequence under the policy $ \bar{\pi} $ when $ x_0 = x $.
    For fixed $ x $, the input sequence $ u_{0:\ft-1} $ under $ \jref^x $ has a density function 
    \begin{align*}
        \bar{r}_x^{(N)} (\bfu) &:= \prod_{k=0}^{N-1} \left[ \Ee^{-q_k (x_k^{x,\bfu})} \frac{\psi_{k+1} (x_{k+1}^{x,\bfu})}{ \psi_k (x_k^{x,\bfu}) } r_k(u_k) \right] \\
        &= \frac{\exp(-Q_x(\bfu))}{\psi_0 (x)} r^{(N)} (\bfu) , \quad \bfu \in \calO^\ft ,
    \end{align*}
    where the first equality follows from \eqref{eq:pibar_2}.
    Since $ 0 < \psi_0 (x) \le 1 $, \eqref{eq:r_J_Q_bound} implies
    \[
        \bar{r}_x^{(N)} (\bfu) \ge \Ee^{-\bar{q}} \underline{r}, \quad \forall x \in X, \ \forall \bfu \in C .
    \]

    Let $ \bar{K} (x,\cdot) $ be the distribution of $ x_\ft $ under $ \jref^x $.
    By definition, for any $ A \in \scrB(\bbR^n) $, we have
    \begin{align}
        \bar{K} (x, A) &= \int_{\calE_x^{-1} (A)} \bar{r}_x^{(N)} (\bfu)\rmd \bfu \nonumber\\
        &\ge \int_{\calE_x^{-1} (A)\cap C \cap \{\calJ_x > 0 \}} \bar{r}_x^{(N)} (\bfu)\rmd \bfu . \nonumber
    \end{align}
    Since $C\subset\calO^\ft$ is compact and $\calE_x$ is $C^1$ on $\calO^\ft$,
its restriction to a neighborhood of $C$ admits a Lipschitz
extension to $\bbR^d$. For $j\ge1$, let
$
    C_j:=C\cap\{\calJ_x\ge1/j\}
$.
Then,
$
    g_j := \frac{\bar r_x^{(\ft)}}{\calJ_x}
    \mathbf 1_{\calE_x^{-1}(A)\cap C_j}
$
is integrable. Hence, the coarea formula
~\cite[Theorem~3.11]{EvansGariepy2015} applies and yields
\begin{align}
&\int_{\calE_x^{-1}(A)\cap C_j}
    \bar r_x^{(\ft)}(\bfu) \rmd\bfu \nonumber\\
&\quad =
\int_A
\int_{\calE_x^{-1}(y)\cap C_j}
    \frac{\bar r_x^{(\ft)}(\bfu)}{\calJ_x(\bfu)}
    \calH^{d-n}(\rmd\bfu) \rmd y .
\label{eq:coarea_truncated}
\end{align}
Since $C_j\uparrow C\cap\{\calJ_x>0\}$, letting $j\to\infty$ and
applying the monotone convergence theorem, we obtain
\begin{align}
&\int_{\calE_x^{-1}(A)\cap C\cap\{\calJ_x>0\}}
    \bar r_x^{(\ft)}(\bfu)\rmd\bfu \nonumber\\
&\quad =
\int_A
\int_{\calE_x^{-1}(y)\cap C\cap\{\calJ_x>0\}}
    \frac{\bar r_x^{(\ft)}(\bfu)}{\calJ_x(\bfu)}
    \calH^{d-n}(\rmd\bfu)\rmd y .
\label{eq:coarea_application}
\end{align}
Therefore, for $ \mu_0 $-almost every $ x \in X $,
    \begin{align}
        \bar{K} (x,A) &\ge \frac{\Ee^{-\bar{q}} \underline{r}}{\bar{\calJ}} \int_{A\cap S} \calH^{d-n} \! \left(\calE_x^{-1} (y) \cap C \cap \{\calJ_x > 0 \} \right) \rmd y \nonumber\\
        &\ge c_0 \mu_L^n (A \cap S),
    \end{align}
    where we used \eqref{eq:Hausdorff_positive} and
    $
        c_0 := \frac{\Ee^{-\bar{q}} \underline{r} h_0}{\bar{\calJ}} > 0
    $.

    Since $ \jref $ has initial marginal $ \mu_0 $, its endpoint distribution disintegrates as
\begin{equation}\label{eq:mu_0_K_Q}
    \bar{\bbQ}_{0,\ft} (\rmd x ,\rmd y) = \mu_0 (\rmd x) \bar{K} (x, \rmd y) .
\end{equation}
    Thus, we have
    \begin{equation}\label{eq:Qbar_lower}
        \bar{\bbQ}_{0,\ft} (\rmd x, \rmd y) \ge c_0 \mu_0 (\rmd x) \bfone_{S} (y) \rmd y .
    \end{equation}
    It follows that $ M(\rmd x, \rmd y) := \mu_0 (\rmd x) \bfone_{S} (y) \rmd y \ll \bar{\bbQ}_{0,\ft} $.
    Moreover, since \(\mu_N(S)=1\) and \(\mu_N(\rmd y)=\rho_N(y)\rmd y\), the coupling \(\gamma^\circ=\mu_0\otimes\mu_N\) satisfies
    \[
        \gamma^\circ (\rmd x, \rmd y) = \mu_0 (\rmd x) \bfone_{S}(y) \rho_\ft (y) \rmd y = \rho_N(y) M(\rmd x, \rmd y) .
    \]
    Therefore, $ \gamma^\circ \ll \bar{\bbQ}_{0,\ft} $. In addition, by the boundedness of $ \rho_\ft $ and \eqref{eq:Qbar_lower}, we have
  \[
    \frac{\rmd \gamma^\circ}{\rmd \bar{\bbQ}_{0,N}}(x,y)
    \le \frac{\rho_N(y)}{c_0}
    \le \frac{\|\rho_N\|_\infty}{c_0},
    \qquad \gamma^\circ\text{-a.s.}
\]
This implies that $ r := \rmd \gamma^\circ / \rmd \bar{\bbQ}_{0,\ft} $ is essentially bounded. Since the function $ s \mapsto s \log s $ is bounded below on $ [0,\infty) $ and bounded above on every bounded interval, $ r \log r $ is integrable with respect to $ \bar{\bbQ}_{0,\ft} $. Hence, we obtain
\[
    \kl{\gamma^\circ}{\bar{\bbQ}_{0,\ft}} = \int r \log r \rmd \bar{\bbQ}_{0,\ft} < \infty .
\]
This completes the proof.
\end{proof}

We now interpret the condition~\eqref{eq:Hausdorff_positive} in terms of the reachability of \eqref{eq:system}.
Recall that $
    \calE_x^{-1}(y)
$
is the set of all input sequences that steer \(x\) to \(y\) at time
\(N\).
The condition
$
    \calJ_x(\mathbf u)>0
$
is equivalent to
$
    \rank D_{\mathbf u}\calE_x(\mathbf u)=n
$.
Thus, a control sequence in
\[
    \calE_x^{-1}(y)\cap C\cap\{\calJ_x>0\}
\]
not only steers \(x\) to \(y\), but does so at a regular point of the
finite-horizon endpoint map.  In other words, infinitesimal
variations of the whole input sequence around such a control can
generate arbitrary first-order variations of the terminal state.
This condition concerns the accumulated input-to-state effect over the
whole horizon.

The compact set \(C\subset\calO^N\) restricts attention to input
sequences that remain in a common compact subset of the interior of
the admissible input set.  Thus, the condition does not rely on
steering inputs approaching the input boundary or escaping to
arbitrarily large values.  This uniform interiority, together with compactness, yields uniform bounds on the reference input density $ r^{(N)} $, the accumulated state cost $ Q_x $, and the Jacobian factor $ \calJ_x $ over the relevant control sequences, which are used in applying the coarea formula.

At a regular point, the level set
\(\calE_x^{-1}(y)\) is locally a \((d-n)\)-dimensional manifold by the regular level set theorem.
The \(k\)-dimensional Hausdorff measure \(\calH^k\) provides a natural notion of
\(k\)-dimensional size for subsets of Euclidean space, including lower-dimensional
sets with zero ambient Lebesgue measure.
Accordingly,
\[
    \calH^{d-n}\!\left(
        \calE_x^{-1}(y)\cap C\cap\{\calJ_x>0\}
    \right)
\]
measures the \((d-n)\)-dimensional size of the family of regular
interior input sequences that steer \(x\) to \(y\).
The condition~\eqref{eq:Hausdorff_positive}
therefore requires this family to have a uniformly nonvanishing size
over the prescribed initial and terminal supports.  This is stronger
than ordinary reachability, which only requires the set
\(\calE_x^{-1}(y)\) to be nonempty.

\begin{remark}[Reachability and finite-cost feasibility]
Point-to-point reachability alone does not guarantee finite-cost feasibility.
For example, consider $N=1$, 
\[
    x_1 = \frac{1}{2} + \frac{1}{\pi} \arctan u_0, \quad \ell_0 (u) = \frac{u^2}{2}, \quad q_0 \equiv 0,
\]
and let $ \mu_1 $ be uniform on $ (0,1) $.
Then, the associated reference $ R_0 $ is the standard Gaussian distribution, while the unique input distribution that steers the state to $ \mu_1 $ is the standard Cauchy distribution. However, the KL divergence between the Cauchy distribution and the Gaussian distribution is infinite.
Hence, distribution steering is feasible, whereas finite-cost MaxEnt distribution steering is not.
\hfill $ \diamondsuit $
\end{remark}

Theorem~\ref{thm:uniform-reachability-finite-entropy} requires every initial state in the support of $ \mu_0 $ to reach every target state on the support of $ \mu_\ft $. This can be restrictive especially under a bounded input set $ \bbU $. The following result relaxes the assumption by partitioning the initial and target supports into pairs of regions carrying the same probability mass and requiring regular interior reachability only within each pair.
The proof is given in Appendix~\ref{app:reachability_block}.
\begin{corollary}[Blockwise reachability]\label{cor:reachability_block}
    Suppose that the assumptions of Theorem~\ref{thm:uniform-reachability-finite-entropy} hold except that the condition~\eqref{eq:Hausdorff_positive} is replaced by the following condition.
    Let $ \{X_i\}_{i=1}^M $ and $ \{S_i\}_{i=1}^M $ be finite Borel partitions of $ X $ and $ S $, respectively, such that
    \begin{equation}\label{eq:mass_identity}
        w_i := \mu_0 (X_i) = \mu_N (S_i) > 0, \quad \forall i \in \bbra{1,M} .
    \end{equation}
    For each $ i \in \bbra{1,M} $, assume that there exist a compact set $ C_i \subset \calO^N $ and a constant $ h_i > 0 $ such that, for $ \mu_0 $-almost every $ x \in X_i $ and $ \mu_L^n $-almost every $ y\in S_i $,
    \begin{equation}\label{eq:block_reachability}
        \calH^{d-n} \! \left( \calE_x^{-1} (y) \cap C_i \cap \{ \calJ_x > 0\} \right) \ge h_i .
    \end{equation}
    Then, Assumption~\ref{ass:finite_kl} holds.
\end{corollary}

\subsection{Fully actuated control-affine systems}\label{subsec:control_affine}
In this subsection, we focus on the following control-affine system
\begin{equation}\label{eq:control_affine}
    f_k(x,u) = \bar{f}_k (x) + g_{k} (x) u, 
\end{equation}
where $ \bar{f}_k : \bbR^n \to \bbR^n $ and $ g_k : \bbR^n \to \bbR^{n\times m} $. We assume that at time $ k = N-1 $, the system is fully actuated, i.e., $ m \ge n $ and
\begin{equation}\label{eq:full_actuation}
    g_{N-1} (x) g_{N-1} (x)^\top \succ 0 , \quad \forall x \in \bbR^n .
\end{equation}
Then, the existence of a coupling with finite KL-divergence is guaranteed as follows.
The proof is deferred to Appendix~\ref{app:control_affine}.
\begin{corollary}[Fully actuated control-affine systems]\label{cor:control_affine}
    Suppose that the assumptions of Theorem~\ref{thm:uniform-reachability-finite-entropy} hold except for \eqref{eq:Hausdorff_positive}, and let $ \bbU = \bbR^m $. Suppose also that $ \{f_k\} $ are given by \eqref{eq:control_affine} and satisfy \eqref{eq:full_actuation}.
    Then, Assumption~\ref{ass:finite_kl} holds.
\end{corollary}

\subsection{Reachable linear systems}
In the rest of this section, we consider the linear system
\begin{equation}\label{eq:linear_sys}
    f_k (x,u) = A_k x + B_k u.
\end{equation}
Then, the endpoint map $ \calE_x $ is affine and its Jacobian is the finite-horizon reachability matrix. This structure yields a more explicit sufficient condition for Corollary~\ref{cor:reachability_block}.
For $ 0\le j \le k \le N $, define the state-transition matrix
\begin{equation}
    \Phi_A (k,j) :=
    \begin{dcases}
        A_{k-1} A_{k-2} \cdots A_j, & k > j, \\
        I, & k = j .
    \end{dcases}
\end{equation}
Define the reachability matrix 
\begin{equation}
    \calR_N := [ \Phi_A (N,1)B_0, \ldots, \Phi_A (N,N)B_{N-1} ] 
\end{equation}
and let $ \bfu = [u_0^\top,\ldots,u_{N-1}^\top]^\top $. Then, 
\begin{equation}\label{eq:endpoint_map_linear}
    \calE_x (\bfu) = \Phi_A (N,0) x + \calR_N \bfu .
\end{equation}

When \(\mathcal R_N\) has full row rank, every endpoint constraint $ (x_0,x_N) = (x,y) $ admits the affine family of steering inputs
\[
\mathcal R_N^\dagger
\bigl(y-\Phi_A(N,0)x\bigr)+\ker\mathcal R_N.
\]
The next result shows that Assumption~\ref{ass:finite_kl} follows if a uniformly thick subset of this family remains in the interior of the admissible input set for each matched pair of endpoint regions.
The proof is shown in Appendix~\ref{app:linear_reachability}.
\begin{corollary}[Reachable linear systems]\label{cor:linear_reachability}
    Suppose that the assumptions of Theorem~\ref{thm:uniform-reachability-finite-entropy} hold except for \eqref{eq:Hausdorff_positive}, and let $ \{f_k\} $ be given by \eqref{eq:linear_sys} with $ \rank \calR_N = n $.
    Let $ \{X_i\}_{i=1}^M $ and $ \{S_i\}_{i=1}^M $ be finite Borel partitions of $ X $ and $ S $ satisfying \eqref{eq:mass_identity}.
    Let
    $
        \calR_N^\dagger := \calR_N^\top (\calR_N \calR_N^\top)^{-1}
    $
    be the Moore--Penrose right inverse, and let $ \calN_{\calR} := \mathop{\rm ker} \calR_N $. Assume that $ \mu_N (\rmd y) = \rho_N (y) \rmd y $, where $ \rho_\ft $ is bounded.
    For each $ i \in \bbra{1,M} $, suppose that there exists $ \delta_i > 0 $ such that
    \begin{align}
        C_i &:= \Bigl\{ \calR_N^\dagger (y - \Phi_A (N,0) x)  + v : \nonumber\\
        &\qquad  x\in \overline{X}_i, \ y\in \overline{S}_i, \ v\in \calN_{\calR}, \ \| v \| \le \delta_i  \Bigr\} \subseteq \calO^N , \label{eq:linear_condition}
    \end{align} 
    where $ \overline{X}_i $ and $ \overline{S}_i $ are the closures of $ X_i $ and $ S_i $, respectively.
    Then, Assumption~\ref{ass:finite_kl} holds.
\end{corollary}

Corollary~\ref{cor:linear_reachability} separates the reachability and input-constraint requirements: full row rank of \(\mathcal R_N\) guarantees regular finite-horizon reachability, while the condition~\eqref{eq:linear_condition} ensures that a uniformly nonvanishing family of steering inputs remains away from the boundary of the admissible input set.

\begin{remark}[Unconstrained input set]
Suppose that \(\bbU=\bbR^m\), so that
\(\calO^N=\bbR^{mN}\), and that \(\mathcal R_N\) has full row rank.
Then, the condition \eqref{eq:linear_condition} is satisfied.
Indeed, let \(X,S\subset\bbR^n\) be compact sets such that
\(\mu_0(X)=\mu_N(S)=1\).
For any \(\delta>0\), define $ C $ by \eqref{eq:linear_condition} with $ M = 1 $.
Since
\(\calO^N=\bbR^{mN}\), we trivially have
\(C\subseteq\calO^N\).
Thus, the condition~\eqref{eq:linear_condition} holds with a single pair
\((X,S)\).
Consequently, under the remaining assumptions of Corollary~\ref{cor:linear_reachability}, in particular the compact-support
and bounded-density assumptions on the endpoint distributions, the reachability $ \rank\mathcal R_N=n $ ensures that Assumption~\ref{ass:finite_kl} holds.
\hfill $ \diamondsuit $
\end{remark}

\section{Conclusion}\label{sec:conclusion}
We established an SB characterization of MaxEnt distribution steering for discrete-time nonlinear systems with general state and input costs. By reformulating the control objective as a KL divergence on the \red{state-input} space, we showed that the optimal controlled state process is an SB and that the bridge can be realized by a Markov control policy. We also interpreted the state cost through killing and Feynman--Kac weighting, and characterized finite-cost feasibility in terms of a finite-KL endpoint coupling. We further derived reachability-based sufficient conditions for this feasibility, together with specializations to fully actuated control-affine and linear systems. A nonlinear numerical example illustrated the resulting SB construction.

Future work includes developing scalable numerical methods for nonlinear and high-dimensional systems, where direct discretization of the state space becomes computationally prohibitive.
It is also of interest to derive less conservative and more readily verifiable conditions for finite-cost feasibility.

\section*{References}

\vspace{-\baselineskip}
\bibliographystyle{IEEEtran}
\bibliography{MaxEnt_generalSB}

%%%%%%%%%%%%%%%%%%%%%%%%%%%%%%%%%%%%%%%%%%%%%%%%%%%%%%%%%%%%%%%%%%%%%%%%%%%%%%%%
\appendices
\section{Proof of Proposition~\ref{prop:performance_decompose}}\label{app:performance_decompose}

Before the proof, we introduce the following lemma.
\begin{lemma}\label{lem:feynman-kac_density}
    The distributions $ \jref $ and $ \bfR $ in \eqref{eq:barP_def} satisfy
    \[
        \jref \ll \bfR, \quad \bfR \ll \jref ,
    \]  
    and
    \begin{equation}\label{eq:dPbar_dR}
        \frac{\rmd \jref}{\rmd \bfR} = \frac{\exp\! \left( - \sum_{k=0}^{\ft-1} q_k (x_k) \right)}{\psi_0 (x_0)} .
    \end{equation}
\end{lemma}
\vspace{0.2cm}
\begin{proof}
    By \eqref{eq:desire_bellman} and \eqref{eq:pibar_def}, we have
    \begin{equation}\label{eq:pibar_2}
            \bar{\pi}_k (\rmd u \mid x) = \Ee^{-q_k (x)} \frac{\psi_{k+1} (f_k(x,u))}{\psi_k (x)} R_k (\rmd u) . 
    \end{equation}
    Therefore, it holds that
    \begin{equation*}
        \frac{\rmd \bar{\pi}_k (\cdot \mid x_k)}{\rmd R_k} (u_k) = \Ee^{-q_k (x_k)} \frac{\psi_{k+1} (x_{k+1})}{\psi_k (x_k)} .
    \end{equation*}
    Substituting this into
    \[
        \frac{\rmd \jref}{\rmd \bfR} = \prod_{k=0}^{\ft-1} \frac{\rmd \bar{\pi}_k (\cdot \mid x_k)}{\rmd R_k} (u_k) ,
    \]
    we obtain \eqref{eq:dPbar_dR}.

    By Assumption~\ref{ass:cost}, it holds that $ \exp (-\sum_{k=0}^{\ft-1} q_k (x_k)) > 0 $ for any $ x_k \in \bbR^n $.
    Combining this with \eqref{eq:psi_bounded} and the expression~\eqref{eq:dPbar_dR} yields
    \[
        0 < \frac{\rmd \jref}{\rmd \bfR} < \infty , \quad \bfR\text{-a.s.},
    \]
    which means $ \jref \ll \bfR $ and $ \bfR \ll \jref $.
\end{proof}

Now, we are ready to prove Proposition~\ref{prop:performance_decompose}.
By repeatedly applying the chain rule for KL divergence~\cite[Theorem~2.4]{Leonard2014some} to separate each input $ u_k $ from history $ h_k $,
and noting that $\bfP^\pi$ and $\bfR$ have the same initial marginal
$\mu_0$, we obtain
\begin{equation}\label{eq:Ppi_R_kl}
\kl{\bfP^\pi}{\bfR}
=
\sum_{k=0}^{N-1}
\bbE^\pi\!\left[
\kl{\pi_k(\cdot\mid h_k)}{R_k}
\right].
\end{equation}
If $ \bfP^\pi \not\ll \bfR $, then $\mathbf P^\pi\not\ll\jref$ because
$\jref \ll\mathbf R$. Thus, both sides of \eqref{eq:performance_decompose} are infinite by \eqref{eq:Ppi_R_kl}.

    In what follows, we suppose that $ \bfP^\pi \ll \bfR $.
    By Lemma~\ref{lem:feynman-kac_density}, we have
    \[
        \log \frac{\rmd \bfR}{\rmd \jref} = \log \psi_0 (x_0) + \sum_{k=0}^{\ft-1} q_k (x_k)  .
    \]
Hence, we obtain
\[
\begin{aligned}
\log\frac{\rmd\mathbf P^\pi}{\rmd\jref}
&=
\log\frac{\rmd\mathbf P^\pi}{\rmd\mathbf R}
+\log\psi_0(x_0) + \sum_{k=0}^{N-1}q_k(x_k).
\end{aligned}
\]
Integrating the above equality with respect to \(\mathbf P^\pi\) yields
\[
\begin{aligned}
\kl{\mathbf P^\pi}{\jref}
&=
\kl{\mathbf P^\pi}{\mathbf R}
-\int_{\bbR^n}V_0(x)\mu_0(\rmd x) \\
&\quad +\bbE^\pi\!\left[\sum_{k=0}^{N-1}q_k(x_k)\right],
\end{aligned}
\]
where the first and third terms are nonnegative and the second term cannot be $ -\infty $ by Assumption~\ref{ass:finite_value_func}.
Combining this with \eqref{eq:Ppi_R_kl}, we obtain
\[
\kl{\mathbf P^\pi}{\jref}
=
J(\pi)
-\int_{\bbR^n}V_0(x)\mu_0(\rmd x) ,
\]
which proves \eqref{eq:performance_decompose}.

Since the first term of \eqref{eq:performance_decompose} does not depend on $ \pi $, Problem~\eqref{prob:unconstrained} is equivalent to minimizing the KL divergence $ \kl{\bfP^\pi}{\jref} $, which attains its minimum value of $ 0 $ if and only if $ \bfP^\pi = \jref $. Moreover, when $ \pi = \bar{\pi} $, we have $ \bfP^\pi = \jref $. 
Hence, $ \bar{\pi} $ is an optimal policy, and $ \jref $ is the unique optimal state-input distribution.
This completes the proof.

\section{Proof of Proposition~\ref{prop:markov}}\label{app:markov}
We show \eqref{eq:state_law_equal} by induction. If \eqref{eq:state_law_equal} holds at time $ k $, then for any bounded Borel function $ \varphi : \bbR^n \rightarrow \bbR $, we have
    \begin{align*}
        \bbE^\pi [\varphi (x_{k+1})] &= \int \int \varphi(f_k(x,u)) (\sfM \pi)_k (\rmd u \mid x) \bbP_k^\pi (\rmd x)\\
        &= \bbE^{\sf M \pi} [\varphi(x_{k+1})] .
    \end{align*}
    Since this equality holds for any bounded Borel function $\varphi$, the distributions of $x_{k+1}$ under $\pi$ and $\sfM\pi$ coincide. Since both processes have the same initial distribution $\mu_0$, induction proves \eqref{eq:state_law_equal} for any $k\in \bbra{0,\ft}$.

    Next, suppose that $ J(\pi) < \infty $ and prove \eqref{eq:J_pi_Mpi}. 
    Since $q_k\ge0$, it holds that
$
\bbE^\pi\!\left[
\kl{\pi_k(\cdot\mid h_k)}{R_k}
\right]<\infty
$
for any $k\in \bbra{0,N-1}$. Hence,
\begin{equation}\label{eq:abs_pi_R}
\pi_k(\cdot\mid h_k)\ll R_k,
\qquad \bbP^\pi\text{-a.s.}
\end{equation}
    For notational simplicity, set
    \[ 
    m_k(\rmd u\mid x):=(\mathsf M\pi)_k(\rmd u\mid x). 
    \] 
and let
$
z_k:=(x_0,u_0,\ldots,x_{k-1},u_{k-1})
$
so that $h_k=(z_k,x_k)$. Let $P_{u_k,x_k,z_k}^\pi := (u_k,x_k,z_k)_\#\bfP^\pi$ be the joint distribution of
$(u_k,x_k,z_k)$ under $\pi$, and define the reference distribution
\[
M_k(\rmd u,\rmd x,\rmd z)
:=R_k(\rmd u)P^\pi_{x_k,z_k}(\rmd x,\rmd z),
\]
where $ P_{x_k,z_k}^\pi := (x_k,z_k)_\# \bfP^\pi $.
The above absolute continuity~\eqref{eq:abs_pi_R} implies $P_{u_k,x_k,z_k}^\pi\ll M_k$.
Under $M_k$, $z_k\to x_k\to u_k$ forms a Markov chain.
Moreover, under $P_{u_k,x_k,z_k}^\pi$, the conditional distributions of $u_k$
given $(x_k,z_k)$ and given $x_k$ are
$\pi_k(\cdot\mid h_k)$ and $m_k(\cdot\mid x_k)$, respectively,
whereas under $M_k$ both conditional distributions equal $R_k$.
Therefore, applying the chain rule for conditional KL divergence
\cite[Lemma~5.5.4]{Gray1990entropy} to $ P_{u_k,x_k,z_k}^\pi $ and $ M_k $ yields
    \begin{align} 
        \bbE^{\pi}\!\left[ \kl{ \pi_k(\cdot\mid h_k)}{R_k}\right] &= \bbE^{\pi}\!\left[ \kl{\pi_k(\cdot\mid h_k) }{m_k(\cdot\mid x_k)}\right] \nonumber\\ 
        &+ \int_{\bbR^n} \kl{m_k(\cdot\mid x)}{R_k} \bbP_k^\pi (\rmd x) . \label{eq:markovization-entropy-decomposition} 
    \end{align}

    By \eqref{eq:state_law_equal}, the distribution of $x_k$ under $\mathbf P^{\mathsf M\pi}$ is also $\bbP_k^\pi$. 
    Hence,
    \begin{align*}
    &\int_{\bbR^n} \kl{m_k(\cdot\mid x)}{R_k} \bbP_k^\pi (\rmd x) \\
    &= \bbE^{{\mathsf M\pi}}\!\left[ \kl{ (\mathsf M\pi)_k(\cdot\mid x_k)}{R_k}\right] 
    \end{align*}
    and
    \[ 
    \bbE^{\pi}[q_k(x_k)] = \bbE^{{\mathsf M\pi}}[q_k(x_k)]. 
    \] 
    Therefore, summing \eqref{eq:markovization-entropy-decomposition} over $k=0,\ldots,N-1$ and adding the equal state-cost terms, we obtain the desired result~\eqref{eq:J_pi_Mpi}.

\section{Proof of Proposition~\ref{prop:opt_policy_condi}}\label{app:opt_policy_condi}
For notational convenience, define $ L (\omega) := G (x_0(\omega), x_\ft (\omega)) $. By the definition of \(\bfP^*\) in \eqref{eq:Pstar_def}, 
$
    L
    =
    \rmd\bfP^* / \rmd\jref,
$
$ \jref $-almost surely.
In particular, \(L\geq 0\) and
\(\mathbb E_{\jref}[L]=1\).

By definition, we have $ \int_\bbU \pi_k^* (\rmd u \mid x) = 1 $ whenever $ 0 < \alpha_k < \infty $, and it equals the reference policy $ \bar{\pi}_k $ elsewhere. 
    Thus, $ \pi_k^* $ is a probability kernel.

    For any bounded nonnegative Borel function $ \zeta : \bbU \rightarrow [0,\infty) $, Bayes' formula for conditional expectations yields
    \begin{equation}\label{eq:zeta_cond}
        \bbE_{\bfP^*} [\zeta (u_k) \mid x_k] = \frac{\bbE_{\jref}[ L \zeta(u_k) \mid x_k ]}{\bbE_{\jref}[ L \mid x_k]} ,
    \end{equation}
    provided that $ 0 < \bbE_{\jref}[ L \mid x_k]  < \infty $ $ \bfP^* $-almost surely.
    By the tower property, 
    \[
        \alpha_k (x_k) = \bbE_{\jref} [ L \mid x_k] , \quad \jref\text{-a.s.} 
    \]
    The positivity and finiteness of $ \alpha_k $ can be seen as follows.
    Since $ L \ge 0 $ and $ \bbE_{\jref} [L] = 1 $, $ \alpha_k (x_k) $ is finite $ \jref $-almost surely. Moreover, since $ \bfP^* \ll \jref $, it holds that
    \[
        \alpha_k (x_k) < \infty, \quad \bfP^*\text{-a.s.}
    \]
    For $ S_0 := \{x\in \bbR^n : \alpha_k(x) = 0 \} $, we have
    \[
        \bfP^* (x_k \in S_0) = \bbE_{\jref} [ L \bfone_{S_0}(x_k) ] = \bbE_{\jref}[\alpha_k (x_k) \bfone_{S_0} (x_k)] = 0.
    \]
    Thus,
    \begin{equation}\label{eq:positive_alpha}
        0 < \alpha_k (x_k) < \infty, \quad \bfP^*\text{-a.s.}
    \end{equation}

    By the tower property, the numerator of \eqref{eq:zeta_cond} can be written as
    \begin{align*}
        \bbE_{\jref} [ L\zeta(u_k) \mid x_k ] &= \bbE_{\jref}[ \zeta(u_k) \beta_k (x_k, u_k) \mid x_k ] \\
        &= \int_\bbU \zeta(u) \beta_k (x_k, u) \bar{\pi}_k (\rmd u \mid x_k) .
    \end{align*}
    Substituting this into \eqref{eq:zeta_cond}, we obtain
    \begin{align*}
    \bbE_{\bfP^*} [\zeta (u_k) \mid x_k] &= \frac{\int_\bbU \zeta(u) \beta_k(x_k,u) \bar{\pi}_k (\rmd u \mid x_k)}{\alpha_k(x_k)} \\
        &= \int_\bbU \zeta(u) \pi_k^* (\rmd u \mid x_k) .
    \end{align*}
    This implies that
    \[
        \pi_k^* (\rmd u \mid x_k) = \bfP^* (u_k \in \rmd u \mid x_k) , \quad \bfP^*\text{-a.s.}
    \]
    Therefore, $ \pi^* $ is the Markovization of an optimal policy generating $ \bfP^* $, and by Theorem~\ref{thm:maxent_solves_SB}, $ \pi^* $ must be an optimal policy.
    This completes the proof.

\section{Proof of Proposition~\ref{prop:doob_opt_policy}}\label{app:doob_opt_policy}

As in the proof of Proposition~\ref{prop:opt_policy_condi}, we show that
\begin{equation}\label{eq:Pstar_pi}
    \mathbf P^*(u_k\in\rmd u\mid h_k)
    =
    \pi_k^*(\rmd u\mid x_k)
    \qquad
    \mathbf P^*\text{-a.s.}
\end{equation}
To this end, we first show that
\begin{equation}
    \bbE_{\jref}[b(x_N)\mid h_k]
    =
    b_k(x_k) , \quad \forall k \in \bbra{0,\ft} .
    \label{eq:h-conditional-expectation}
\end{equation}
For $k=N$, this follows immediately from
$b_N=b$.  Suppose that
\eqref{eq:h-conditional-expectation} holds at time $k+1$.
By the tower property,
\[
\begin{aligned}
    \bbE_{\jref}[b(x_N)\mid h_k]
    &=
    \bbE_{\jref}\!\left[
        \bbE_{\jref}[b(x_N)\mid h_{k+1}]
        \,\middle|\, h_k
    \right]\\
    &=
    \bbE_{\jref}\!\left[
        b_{k+1}(x_{k+1})
        \,\middle|\, h_k
    \right] \\
    &=
    \int_{\bbU}
        b_{k+1}(f_k(x_k,u))
        \bar\pi_k(\rmd u\mid x_k) =
    b_k(x_k) ,
\end{aligned}
\]
where the last equality follows from \eqref{eq:markov-h}.
By induction, \eqref{eq:h-conditional-expectation} holds for any
$k\in \bbra{0,N}$.

Since $x_0$ is measurable with respect to the $ \sigma $-algebra generated by $ h_k $, it follows that
\begin{equation}
\begin{aligned}
    \bbE_{\jref}[L\mid h_k]
    = a(x_0) \bbE_{\jref}[b(x_\ft) \mid h_k] =
    a(x_0) b_k(x_k).
\end{aligned}
\label{eq:L-given-Fk}
\end{equation}
By the same argument as for \eqref{eq:positive_alpha}, $\bbE_{\jref}[L\mid h_k]$ is positive and finite
$\mathbf P^*$-almost surely.
Together with the assumption~\eqref{eq:initial-factor-positive}, this implies
\[
    0< b_k(x_k)<\infty,
    \quad
    \mathbf P^*\text{-a.s.}
\]

For any bounded Borel function $\zeta:\bbU\to[0,\infty)$,
Bayes' formula for conditional expectations gives
\begin{equation}
    \bbE_{\mathbf P^*}[\zeta(u_k)\mid h_k]
    =
    \frac{
        \bbE_{\jref}[L\zeta(u_k)\mid h_k]
    }{
        \bbE_{\jref}[L\mid h_k]
    },
    \qquad
    \mathbf P^*\text{-a.s.}
    \label{eq:Bayes-corollary}
\end{equation}
By
\eqref{eq:L-given-Fk} at time $k+1$,
\[
\begin{aligned}
    \bbE_{\jref}[L\mid h_k,u_k]
    &=
    \bbE_{\jref}[L\mid h_{k+1}]\\
    &=
    a(x_0)b_{k+1}(x_{k+1})\\
    &=
    a(x_0)
    b_{k+1}(f_k(x_k,u_k)).
\end{aligned}
\]
Therefore, using the tower property, the numerator of \eqref{eq:Bayes-corollary} can be written as
\[
\begin{aligned}
    &\bbE_{\jref}[L\zeta(u_k)\mid h_k]\\
    &=
    \bbE_{\jref}\!\left[
        \zeta(u_k)
        \bbE_{\jref}[L\mid h_k,u_k]
        \,\middle|\,h_k
    \right]\\
    &=
    a(x_0)
    \int_{\bbU}
        \zeta(u)
        b_{k+1}(f_k(x_k,u))
        \bar\pi_k(\rmd u\mid x_k).
\end{aligned}
\]
Substituting this equality and \eqref{eq:L-given-Fk} into
\eqref{eq:Bayes-corollary}, we obtain
\[
\begin{aligned}
    \bbE_{\mathbf P^*}[\zeta(u_k)\mid h_k]
    &=
    \int_{\bbU}
        \zeta(u)
        \frac{
            b_{k+1}(f_k(x_k,u))
        }{
            b_k(x_k)
        }
        \bar\pi_k(\rmd u\mid x_k)\\
    &=
    \int_{\bbU}
        \zeta(u)
        \pi_k^*(\rmd u\mid x_k) .
\end{aligned}
\]
Since this holds for any bounded nonnegative Borel function
$\zeta$, \eqref{eq:Pstar_pi} holds.
Using the tower property again, 
\begin{align*}
    \mathbf P^*(u_k\in\rmd u\mid x_k)
    &= \bbE_{\bfP^*} [ \bfP^* (u_k \in \rmd u \mid h_k) \mid x_k ] \\
    &=
    \pi_k^*(\rmd u\mid x_k),
    \quad
    \mathbf P^*\text{-a.s.}
\end{align*}
Therefore, \(\{\pi_k^*\}_{k=0}^{\ft-1}\) is the Markovization of an optimal causal
policy generating \(\mathbf P^*\).  By
Theorem~\ref{thm:maxent_solves_SB}, its Markovization is also optimal.
This completes the proof.

\section{Proof of Proposition~\ref{prop:factorization_sufficient}}\label{app:factorization_sufficient}

First, we show that $ c_S > 0 $.
By Assumption~\ref{ass:finite_kl}, there exists
$\widetilde{\gamma}\in\Gamma(\mu_0,\mu_N)$ such that
$
    \kl{\widetilde{\gamma}}{\bar{\mathbb Q}_{0,N}}<\infty 
$.
Hence,
$
    \widetilde{\gamma}\ll\bar{\mathbb Q}_{0,N}
$.
Since $\mu_N(S)=1$ and the second marginal of
$\widetilde{\gamma}$ is $\mu_N$, we have
$
    \widetilde{\gamma}(\mathbb R^n\times S)=1
$.
Therefore,
\[
    c_S:=\bar{\mathbb Q}_{0,N}(\mathbb R^n\times S)>0.
\]

Since every $\gamma\in\Gamma(\mu_0,\mu_N)$ is concentrated on
$\bbR^n\times S$, we have
\[
    \kl{\gamma}{\bar{\bbQ}_{0,N}^{S}}
    =
    \kl{\gamma}{\bar{\bbQ}_{0,N}}+\log c_S.
\]
Hence, the static SB problem with the reference $ \bar{\bbQ}_{0,\ft} $ and the one with $ \bar{\bbQ}_{0,\ft}^S $ have the
same minimizer $\gamma^*$.  By Assumption~\ref{ass:finite_kl} and Proposition~\ref{prop:static_SB}, the latter problem has a finite minimum.  Therefore, under the assumption~\eqref{eq:mutual_continuity}, we can apply \cite[Theorem~2.1]{Nutz2021introduction}; there exist real-valued measurable
functions $\phi$ and $\chi$ such that
\[
    \frac{\rmd\gamma^*}{\rmd\bar{\bbQ}_{0,N}^{S}}(x,y)
    =\Ee^{\phi(x)+\chi(y)}.
\]
Since
\[
    \frac{\rmd\bar{\bbQ}_{0,N}^{S}}{\rmd\bar{\bbQ}_{0,N}}(x,y)
    =c_S^{-1}\bfone_S(y),
\]
we obtain
\[
    \frac{\rmd\gamma^*}{\rmd\bar{\bbQ}_{0,N}}(x,y)
    =
    \Ee^{\phi(x)}
    \bigl(c_S^{-1}\bfone_S(y)\Ee^{\chi(y)}\bigr).
\]
Thus, \eqref{eq:endpoint_factorization} holds with
$a=\Ee^\phi$ and
$b=c_S^{-1}\bfone_S \Ee^\chi$, and
$0<a<\infty$ $\mu_0$-almost everywhere.
This completes the proof.

\section{Proof of Corollary~\ref{cor:reachability_block}}\label{app:reachability_block}

Fix $i\in\bbra{1,M}$.
Since $X$ and $C_i$ are compact, the same compactness and coarea
arguments as in the proof of Theorem~\ref{thm:uniform-reachability-finite-entropy} yield
a constant $c_i>0$ such that, for $\mu_0$-almost every
$x\in X_i$ and every Borel set $A\subseteq\mathbb R^n$,
    \begin{equation}\label{eq:Kbar_bound}
        \bar{K} (x,A) \ge c_i \mu_L^n (A \cap S_i)  .
    \end{equation}

Define
\[
    \gamma^\circ(\rmd x,\rmd y)
    :=
    \sum_{i=1}^M
    \bfone_{X_i}(x)\mu_0(\rmd x)
    \frac{\bfone_{S_i}(y)}{w_i}\mu_N(\rmd y).
\]
Since $\{X_i\}_{i=1}^M$ and $\{S_i\}_{i=1}^M$ are partitions
and
\(
\mu_0(X_i)=\mu_N(S_i)=w_i
\),
we have
$
    \gamma^\circ\in\Gamma(\mu_0,\mu_N).
$

For each $i$, \eqref{eq:Kbar_bound} yields
\[
    \bar{K}(x,dy)
    \ge c_i\bfone_{S_i}(y) \rmd y
\]
for $\mu_0$-almost every $x\in X_i$.
Moreover, by \eqref{eq:mass_identity} and
\(\mu_N(\rmd y)=\rho_N(y)\rmd y\),
\[
    \gamma^\circ(\rmd x,\rmd y)
    =
    \frac{\rho_N(y)}{w_i}\mu_0(\rmd x)\rmd y
\]
on $X_i\times S_i$, up to null sets. Consequently,
\[
    \frac{\rmd\gamma^\circ}{\rmd\bar{\bbQ}_{0,N}}(x,y)
    \le
    \frac{\rho_N(y)}{w_i c_i}
    \le
    \frac{\|\rho_N\|_\infty}{w_i c_i}
\]
for $\gamma^\circ$-almost every
$(x,y)\in X_i\times S_i$.
Since there are only finitely many blocks and $w_i,c_i>0$,
$\rmd\gamma^\circ/\rmd\bar{\bbQ}_{0,N}$ is essentially bounded.
As in the proof of Theorem~\ref{thm:uniform-reachability-finite-entropy}, this implies $ \kl{\gamma^\circ}{\bar{\bbQ}_{0,\ft}} < \infty $. This completes the proof.

\section{Proof of Corollary~\ref{cor:control_affine}}\label{app:control_affine}

We show that the condition~\eqref{eq:Hausdorff_positive} holds.
    Let $ p := m(N-1) $, $ q := m - n $, and fix $ \delta > 0 $ and a compact set $ D \subset \bbR^p $ with $ \mu_L^p (D) > 0 $.
    For $ \bfv = (u_0,\ldots,u_{N-2}) \in D $, set $ z_x (\bfv) := x_{N-1}^{x,\bfv} $, where $ x_{N-1}^{x,\bfv} $ is defined as in \eqref{eq:state_u}. We also define
    \[
        s(z,y) := g_{N-1} (z)^\top \! \left(g_{N-1} (z) g_{N-1} (z)^\top \right)^{-1} \! \left(y - \bar{f}_{N-1} (z)\right) ,
    \]
    which is a control input $ u_{\ft-1} $ that moves $ x_{N-1} = z $ to $ x_N = y $.
    Since $ X $, $ S $, and $ D $ are compact, $ s(z_x (\bfv), y) $ is uniformly bounded. Hence, there exists a compact $ C \subset \bbR^{mN} = \calO^N $ such that for any $ (x,y) \in X \times S $,
    \begin{align*}
        F_{x,y} &:= \{ (\bfv, s(z_x(\bfv),y) + w) : \\
        &\qquad \bfv \in D, \ w \in \ker g_{N-1} (z_x (\bfv)), \ \| w\|\le \delta  \} \subset C .
    \end{align*}
    By construction,
    $
        F_{x,y} \subset \calE_x^{-1} (y) 
    $.
    Moreover, the last block of $ D_{\bfu} \calE_x  $ in \eqref{eq:reachability_matrix} is $ g_{N-1} (x_{N-1}) $, which has full row rank. Hence, $ \calJ_x ( \bfu) > 0 $ for any $ \bfu \in F_{x,y} $, and
    \[
        F_{x,y} \subset \calE_x^{-1} (y) \cap C \cap \{\calJ_x > 0 \} .
    \]

    It remains to show that \( \calH^{d-n}( F_{x,y} ) > 0\).
    Let
$
    \Pi(\bfv,u_{N-1}):=\bfv 
$.
Since
\(\dim\ker g_{N-1}=q\), for every fixed \(\bfv\in D\),
\(F_{x,y}\cap\Pi^{-1}(\bfv)\) is a translate of a
\(q\)-dimensional ball of radius \(\delta\). 
Moreover, since $\Pi$ is $1$-Lipschitz, Eilenberg's inequality
~\cite[Theorem~2.10.25]{Federer1969geometric} yields a constant
$c_{p,q}>0$, depending only on $p$ and $q$, such that
\[
    \calH^{d-n}(F_{x,y})
    \ge c_{p,q}\,\mu_L^p(D) \omega_q\delta^q >0,
\]
where \(\omega_{q}\) denotes the volume of the unit ball in \(\mathbb R^{q}\), with $ \omega_0 := 1 $.
Therefore, the condition~\eqref{eq:Hausdorff_positive} holds.
The conclusion follows from
Theorem~\ref{thm:uniform-reachability-finite-entropy}.

\section{Proof of Corollary~\ref{cor:linear_reachability}}\label{app:linear_reachability}

Fix \(i\in \bbra{1,M}\).  Since \(\overline{X}_i\) and \(\overline{S}_i\) are compact and
$
\{v\in\calN_{\calR}:\|v\|\le\delta_i\}
$
is compact, the set \(C_i\) defined in
\eqref{eq:linear_condition} is compact.  By assumption,
\(C_i\subseteq\calO^N\).

Now fix \(x\in X_i\) and \(y\in S_i\), and set
\[
    z:=y-\Phi_A(N,0)x.
\]
Since \(\rank\mathcal R_N=n\), the right inverse
\(\mathcal R_N^\dagger\) satisfies
$
    \mathcal R_N\mathcal R_N^\dagger=I.
$ 
Hence, for any \(v\in\calN_{\calR}=\ker\mathcal R_N\), we have
$
    \calE_x(\mathcal R_N^\dagger z+v)
    =y
$.
Thus,
\[
    \mathcal R_N^\dagger z+v
    \in\calE_x^{-1}(y),
    \quad
    \forall v\in\mathcal N_{\calR}.
\]
Moreover, if \(\|v\|\le\delta_i\), then the definition of
\(C_i\), together with \(x\in X_i\) and \(y\in S_i\), gives
$
    \mathcal R_N^\dagger z+v\in C_i
$.
Consequently,
\[
    \mathcal R_N^\dagger z
    +\{v\in\mathcal N_{\calR}:\|v\|\le\delta_i\}
    \subseteq
    \calE_x^{-1}(y)\cap C_i.
\]

Moreover, by \eqref{eq:endpoint_map_linear}, for any $ \bfu \in  \calO^N $, we have
$
    D_{\mathbf u}\calE_x(\mathbf u)=\mathcal R_N 
$.
Since \(\rank\mathcal R_N=n\),
\[
    \calJ_x(\mathbf u)
    =
    \sqrt{\det(\mathcal R_N\mathcal R_N^\top)}>0 ,  \quad \forall \bfu \in \calO^N .
\]
Consequently,
\begin{equation}\label{eq:include}
    \mathcal R_N^\dagger z
    +\{v\in\mathcal N_{\calR}:\|v\|\le\delta_i\}
    \subseteq
    \calE_x^{-1}(y)\cap C_i\cap\{\calJ_x>0\}.
\end{equation}

By the rank--nullity theorem, we have
$
    \dim\mathcal N_{\calR}=d-n
$.
Hence, the set
\(\{v\in\mathcal N_{\calR}:\|v\|\le\delta_i\}\)
is a \((d-n)\)-dimensional closed ball.  By \eqref{eq:include} and translation invariance
of the Hausdorff measure,
\[
\begin{aligned}
&\calH^{d-n}\!\left(
    \calE_x^{-1}(y)\cap C_i\cap\{\calJ_x>0\}
\right) \\
&\ge
\calH^{d-n}\!\left(
    \{v\in\mathcal N_{\calR}:\|v\|\le\delta_i\}
\right)
=:h_i>0.
\end{aligned}
\]
The right-hand side \(h_i\) depends only on \(\delta_i\) and
\(\mathcal N_{\calR}\), and is therefore independent of \(x\in X_i\)
and \(y\in S_i\).  Thus the condition
\eqref{eq:block_reachability} holds for every
\(x\in X_i\) and \(y\in S_i\).
Applying Corollary~\ref{cor:reachability_block} completes the proof.

\end{document}